\documentclass[12pt]{amsart}
\usepackage{amsfonts}
\usepackage{amssymb}
\usepackage{pict2e}
\usepackage{tikz}  
\usepackage{color}
\usepackage{graphicx}

\font\got=eufm10

\def\g2{ \hbox{\got g}_2}
\def\e7{ \hbox{\got e}_7}
\def\e8{ \hbox{\got e}_8}
\def\id{\mathop{\hbox{\rm id}}}

\def\L{{\mathcal L}}
\def\J{{\mathcal J}}

\def\a{\alpha}
\def\b{\beta}
\def\s{\sigma}
\def\tr{\mathop{\rm tr}}
\def\T{\mathcal{T}}

\def\A{\mathcal A}

\def\supp{\mathop{\hbox{\rm  Supp}}}
\def\M{\mathcal M}

\def\span#1{\langle#1\rangle}

\def\H{{\mathbb  H}}

\def\OO{\mathbb{O}}
\def\C{\mathbb C}
\def\Z{\mathbb Z}
\def\F{\mathbb F}
\def\R{\mathbb R}
\def\imag{\textbf{i}}
\newcommand{\ii}{\textbf{i}}
\def\aut{\mathop{\rm Aut}}
\def\Hom{\mathop{\rm Hom}}
\def\der{\mathop{\rm Der}}

\def\Ad{\mathop{\rm Ad}}

\def\ad{\mathop{\rm ad}}

\def\sll{\mathop{\rm sl}}
\def\fix{\mathop{\rm fix}}
\def\dim{\mathop{\rm dim}}

\def\sp31{\mathop{ \mathfrak{sp}}_{3,1}(\mathbb{H})}
\def\sll3{\mathop{ \mathfrak{sl}}(3,\mathbb{O})}

\def\Mat{\mathop{\rm Mat}}

\def\a{\alpha}
\def\si{\sigma}
\def\sign{\mathop{\hbox{\rm sign}}}

\def\tr{\mathop{\rm tr}}

\def\diag{\mathop{\rm diag}}
 \def\e6{  \mathfrak{e}_6 }
 \def\f4{  \mathfrak{f}_4 }
  \def\d4{  \mathfrak{d}_4 }
 \def\c4{  \mathfrak{c}_4 }
 \def\sig#1{ \mathfrak{e}_{6,-#1} }

\newtheorem{te}{Theorem}
\newtheorem{pr}{Proposition}
\newtheorem{lm}{Lemma}
\newtheorem{co}{Corollary}
\newtheorem{re}{Remark}
\title[Fine gradings on $\sig{14}$  ]{ Gradings on the real form $\sig{14}$}  
\author[C. Draper]{Cristina Draper${}^*$}
\thanks{${}^*$\ Partially supported by MCYT grant   MTM2016-76327-C3-1-P and by   the Junta de Andaluc\'{\i}a PAI
project  FQM-336.}
\address{Cristina Draper Fontanals: Departamento de
Matem\'atica Aplicada\\ Escuela de  Ingenier\'{\i}as Industriales\\ Ampliaci\'{o}n Campus de Teatinos, S/N, 29071 M\'alaga,
Spain\\cdf@uma.es}
\author[V. Guido]{Valerio Guido}
\address{Valerio Guido: valiero84@libero.it}

\begin{document}

\setlength{\unitlength}{0.06in}

\maketitle


\begin{abstract}
Six fine gradings on the real form $\sig{14}$ are described,   precisely those ones coming from fine gradings on the complexified algebra. The universal grading groups are $\Z_2^3\times\Z_3^2$, $\Z_2^6$, $\Z\times\Z_2^4$, $\Z_2^7$,
$\Z\times\Z_2^5$ and $\Z^2\times\Z_2^3$.
\end{abstract}


\section{Introduction}

This work continues a series of papers devoted to describing the fine gradings on the  exceptional real Lie algebras. 
The interest of gradings  {has}  
been present in the theory of Lie algebras from the very beginning.   Let us recall that, at the end of the nineteenth century, W.~Killing   used the root decomposition (which is, in fact, a grading over the free-torsion abelian group $\mathbb{Z}^l$ for $l$ the rank of the algebra) in order to classify complex finite-dimensional semisimple Lie algebras.  More recently, the  {physicist} 
J.~Patera and his collaborators proposed to  {start} a systematic study of the gradings on Lie algebras in \cite{LGI}. This seminal work was followed by some others for real classical Lie algebras, as \cite{LGIII} or \cite{Svobodova}. Since then, many authors  {shifted}  
the attention to the complex case.  A remarkable amount of these works was compiled in  \cite{libro}. This  monograph was published in 2013 and it contains almost completely the classification of the gradings on the complex simple finite-dimensional Lie algebras.  Precisely,  the classification of the gradings on type-$E$ algebras is only conjectured in the book,
and had to wait until  \cite{e6}, \cite{nuestroe8} (finite groups) and \cite{Yu} to be completed. A review of the fine gradings on  Lie algebras of type-$E$ is \cite{proceedingsLecce}, which tries to give a  version as unified as possible of the gradings on the three complex Lie algebras of type $E$. Having finished the complex case,   it is the moment to return to the real case, which is   relevant for many applications. The problem of getting a classification of the gradings on real simple Lie algebras is being tackled by several authors almost simultaneously. On one hand, some recent papers devoted to the classical simple real Lie algebras are \cite{reales1}, \cite{reales2} and \cite{reales3}. On the other hand, as mentioned, this is the third of our papers about exceptional real Lie algebras: first, in 
 \cite{reales}, we classified fine gradings on the real algebras of types $G_2$ and $F_4$, and second, in \cite{sig26}, we   described the fine  gradings on the   algebra $\mathfrak{e}_{6,-26}$ coming from fine gradings on the complex algebra $\mathfrak{e}_6$.
Let us recall that there are three non-split and non-compact real algebras of type $E_6$, characterized by the signatures of their Killing forms, namely, -26, -14 and 2. Here we focus    on  $\mathfrak{e}_{6,-14}$, an algebra   
 with frequent appearances  in Physics (see below).  {The difficulty in defining explicitly   $\mathfrak{e}_{6,-14}$ renders extremely useful the knowledge of its structure, gradings, and different models. } 
Such gradings and models could be potentially related with different physical phenomena.
The gradings on $\mathfrak{e}_{6,-14}$ have nothing to do with the gradings on $\mathfrak{e}_{6,-26}$ obtained in \cite{sig26}.
 Indeed, our main result, Theorem~\ref{th_main}, states that there are 6 fine gradings (and essentially only 6) on $\mathfrak{e}_{6,-14}$, whose  universal grading groups are
 $\Z_2^3\times\Z_3^2$, $\Z_2^6$, $\Z\times\Z_2^4$, $\Z_2^7$,
$\Z\times\Z_2^5$ and $\Z^2\times\Z_2^3$. In contrast, the fine gradings occurring on $\mathfrak{e}_{6,-26}$  are 4, with universal grading groups  $\Z_2^6$, $\Z\times\Z_2^4$, $\Z^2\times\Z_2^3$ and  $\Z_4\times\Z_2^4$. Only two pairs of those gradings have the same origin, i.e., isomorphic complexifications (the two $\Z^2\times\Z_2^3$-gradings have quite different properties).
 {Regarding}    
the tools used here for  finding the gradings on $\mathfrak{e}_{6,-14}$, they {are indeed}  
related  {to}   
the tools in \cite{sig26}, but new approaches have been necessary too.  {For each grading, we have tried to do a self-contained treatment, providing a related suitable model of the   algebra $\mathfrak{e}_{6,-14}$.} \smallskip

 The relation between graded Lie algebras and Physics is well-known, see, for instance, the classical paper \cite{fisicos} of 1975.
 Some recent reference is \cite{Atanasov},
    devoted  to  the  applications  of graded Lie algebras for  studying
dynamic symmetries of atomic nucliei.
A survey of algebraic methods  {that}  
are widely used in nuclear
and molecular physics is \cite{Iachello}, which states a spectrum generating algebra as the basic mathematical tool of all such methods, and presents some models where such spectrum generating algebra is just a graded Lie algebra. 
   More applications of the gradings on Lie algebras  to mathematical physics  and
particularly to particle physics are referenced in  our previous work \cite{sig26}. To illustrate them with some examples, recall first the Pauli matrices, which provide a fine $\mathbb{Z}_2^2$-grading on $\mathfrak{su}_2$.  In quantum mechanics, these matrices occur in the Pauli equation which takes into account the interaction of the spin of a particle with an external electromagnetic field. 
On the other hand, the Gell-Mann matrices  provide a  fine  $\mathbb{Z}_2^3$-grading on the compact Lie algebra $\mathfrak{su}_3$
(see Section~\ref{todoZ2}). The Gell-Mann matrices are used in the study of the strong interaction in particle physics, and they  serve to study the internal (color) rotations of the gluon fields associated with the colored quarks of quantum chromodynamics. Also
 the generalized Pauli matrices of order 3  provide a $\mathbb{Z}_3^2$-grading on $\mathfrak{su}_3^\mathbb C$, but note that it has been  necessary to complexify here, since the only groups being the grading groups of a compact Lie algebra are the direct product of copies of $\mathbb Z_2$. \smallskip

Moreover, the role of the Lie algebras of type $E_6$ in Physics is also quite relevant \cite{Wybourne}, giving rise to a large number of references, some of them appearing  
in \cite{sig26}. We can add, for instance, \cite{fis3},  on the discovery limits of different $E_6$ models at some colliders (as Tevatron); \cite{fis4}, on order-two twisted D-branes of
$E_6$; \cite{fis5}, which proposes a model  of dark energy and dark matter based on $E_6$  very different from the usual $E_6$-unification encountered in the literature; 
\cite{fis8}, which proposes $U_{28}\to  SU_{27}\times U_1\to E_6\to G_2\to SO_3$ for the extension of the interacting boson model;
\cite{fis7}, on $N=8$ black holes  
in five dimensions where
the common symmetry is $E_6$;
or \cite{fis6},  which analyses   heterotic string compactifications on specific classes of
manifolds with $SU(3)$-structure, proving that  $E_6$ is still the resulting gauge group in
four dimensions.   
\smallskip

Finally, we would also like  to mention   the growing number of specific appearances of the real form $\mathfrak{e}_{6,-14}$ in the literature.
This algebra appears as U-duality algebra when considering the theories coupled to gravity for $N=5$ supersymmetries and $D=3$ spacetime dimensions (\cite[\S7.3]{motivacion14}).
Also, according to \cite[Proposition~7.1]{fibradosHiggs},
 the Toledo invariant  of $\mathfrak{e}_{6,-14}$ corresponding to the Hitchin-Kostant-Rallis section for the moduli
space of its Higgs bundles is zero. 
 Relative to incidence geometry,
a recent paper concerned with the group $E_{6,-14}$    is \cite{realquadrangle}. It describes a projective embedding (called a Veronese embedding) of the building of the group  $E_{6,-14}$ (a generalized quadrangle)
over the reals, trying to follow the spirit of Freudenthal's description of the Cayley plane.
The paper \cite{Dobrev2}  reviews the progress of the   classification and
construction of invariant differential operators for non-compact semisimple Lie groups, including many references about the importance of these invariant differential operators in the description of physical symmetries. As the study   of such  classification should  proceed 
group by group, the author chooses some groups.  This choice, which includes $E_{6,-14}$, is supported by these groups being non-compact groups  that have a discrete series of representations (the  rank of the group $G$ coincides with the rank of the  maximal compact subgroup $K$), and, besides, $G/K$ belongs to the list of   Hermitian symmetric spaces (each one admits a complex structure which is invariant by the
group of isometries).  
The two exceptional cases in such list (the list of possible groups $G$ with $G/K$ irreducible non-compact Hermitian symmetric space)
are $E_{6,-14}$ and $E_{7,-25}$,
corresponding to   the only Lie  exceptional real algebras with highest weight representations. 
 The same author devotes the work  \cite{Dobrev} to study invariant differential operators focused on $E_{6,-14}$.
 Amazingly, the symmetric space $E_{7,-25}/E_{6,-14}$   appears in  \cite{agujerosnegros} corresponding to the orbit of a large black hole   with central charge equal to zero. 
It is just one of the  three species of the solutions to the black hole attractor equations giving rise to
different mass spectra of the scalar fluctuations. These authors have devoted a considerable amount of papers to study exceptional Lie algebras in relation to physical theories. In their own words (\cite{Marraniexcepcional}):
\lq\lq While describing the results of our recent work on exceptional Lie and Jordan algebras, so tightly intertwined in their connection with elementary particles, we will try to stimulate a critical discussion on the nature of spacetime and indicate how these algebraic structures can inspire a new way of going beyond the current knowledge of fundamental physics.\rq\rq
 \smallskip

\emph{Outline}. In Section~\ref{prelim}, we summarize the basics about gradings on a Lie algebra, while recalling the classification of the fine gradings on $\mathfrak{e}_{6}$ obtained in \cite{e6}. There are 14 fine gradings on the complex Lie algebra $\mathfrak{e}_{6}$, and some invariants are provided, including the universal groups. Also, the necessary  background   about real forms is recalled, and
some notations and results about signatures are stated. In 
Section~\ref{deAlbert}, two gradings on $\mathfrak{e}_{6,-14}$ over the groups $\Z_2^6$ and $\Z\times\Z_2^4$  are constructed based on a model of $\mathfrak{e}_{6,-14}$ as a sum of the set of traceless elements of certain Albert algebra with the Lie algebra of derivations of such Albert algebra (a non-split and non-compact real form of $\mathfrak{f}_4$).  Section~\ref{deTits} provides a  $\Z_2^3\times\Z_3^2$-grading  on $\mathfrak{e}_{6,-14}$ based on the Tits' construction. This requires to determine the signatures of the real forms obtained when applying Tits' construction to different Jordan and composition algebras. 
In Section~\ref{flag}, we exhibit a model of $\mathfrak{e}_{6,-14}$ based on a contact $\mathbb Z$-grading, in which two gradings over the groups $\Z\times\Z_2^5$ and $\Z^2\times\Z_2^3$ are considered. The last fine grading is studied in Section~\ref{todoZ2}. Similarly to the $\Z_2^2$-grading on $\mathfrak{su}_2$ and to the $\Z_2^3$-grading on $\mathfrak{su}_3$, related to the Pauli matrices and to the Gell-Mann matrices respectively, we find now a $\Z_2^7$-grading on $\mathfrak{e}_{6,-14}$.  This provides a basis of $\mathfrak{e}_{6,-14}$ with interesting properties (Corollary~\ref{co_basemona}). Most of the real simple (finite-dimensional) Lie algebras of rank $l$ admit a fine grading over the group $\Z_2^{l+1}$, with very few exceptions, as for instance $\mathfrak{e}_{6,-26}$ which is not $\Z_2^7$-graded \cite[Proposition~11]{sig26}. The last section is devoted to prove that there do not exist  fine gradings on $\mathfrak{e}_{6,-14}$ whose complexification is any of the remaining 8 fine gradings on the complex algebra $\mathfrak{e}_{6}$. This is a case-by-case proof for some of the gradings on $\mathfrak{e}_{6}$, based on the knowledge of their features and of \emph{compatible} subalgebras.   
Some representation theory of real Lie algebras is   used too. We finish with some open problems: apart from the mathematical questions  (not difficult to  {tackle}  although very technical),   we would feel deeply interested in finding the physical meaning of the obtained gradings, specially those ones involving groups with 3-torsion.


 \section{Preliminaries}\label{prelim}
 
We review now the basic concepts about gradings and real forms. 
Most of them appear in \cite{sig26}, but we sketch them here for self-containedness.
 
 \subsection{About gradings}
 The main reference about the topic of gradings on Lie algebras is the monograph \cite{libro}.

 Let $\A$ be a finite-dimensional algebra  over a field  $\F\in\{\R,\C\}$, and $G$ an abelian group.
  A \emph{$G$-grading} $\Gamma$ on $\A$ is a vector space decomposition
\begin{equation}\label{def_gr}
 \Gamma: \A = \bigoplus_{g\in G} \A_g
 \end{equation}
such that
$
 \A_g \A_h\subset \A_{g+h}$ for all $g,h\in G$.
 The subspace $\A_g$ is called \emph{homogeneous component of degree $g$} and its elements 
 \emph{homogeneous elements of degree $g$}. 
 The \emph{support} of the grading is the set $\supp \Gamma :=\{g\in G: \A_g\neq 0\}$.
 The \emph{type} of   $\Gamma$ is the sequence of integers $(h_1,\ldots, h_r)$, where $h_i$ is the number of homogeneous components of dimension $i$, with $i=1,\ldots, r$ and $h_r\neq 0$. Obviously, $\dim \mathcal{A} = \sum_{i=1}^r ih_i$.

  If $\Gamma\colon \A=\oplus_{g\in G} \A_g$ and $\Gamma'\colon \A=\oplus_{h\in H} \A'_{h}$ are gradings over two abelian groups $G$ and $H$, $\Gamma$ is said to be a \emph{refinement} of $\Gamma'$ (or $\Gamma'$ a \emph{coarsening} of $\Gamma$) if for any  $g\in G$, there is $h\in H$ such that $\A_g\subset \A'_{h}$.  A refinement is \emph{proper} if some inclusion $\A_g\subset \A'_{h}$ is proper. A grading is said to be \emph{fine} if it admits no proper refinement.
 Also, $\Gamma$ and $\Gamma'$ are said to be \emph{equivalent} if there is  a bijection $\alpha\colon  \supp \Gamma \rightarrow \supp \Gamma'$  and  $\varphi\in\aut(\A) $ such that $\varphi(\A_s)=\A'_{\alpha(s)}$ for all $s\in \supp \Gamma$. 
Note that in this case the grading groups $G$ and $H$ are not necessarily isomorphic groups. 
For any group grading $\Gamma$ on $\A$, there exists a distinguished group   among the ones which are  grading groups of gradings on $\A$ equivalent to $\Gamma$ such that they are generated by the support of the grading (more details in \cite{libro}). It is usually called the \emph{universal (grading) group of} $\Gamma$. 

If the ground field is the complex field, there is a close relationship between automorphisms and gradings.
First, it is obvious that  every diagonalizable automorphism $f\in\aut(\A)$ produces a grading on $\A$ over the subgroup of $\C^\times$ generated by the spectrum $\mathrm{Spec}(f)$, and also that   any collection of commuting diagonalizable automorphisms produces a grading where the homogeneous components are the simultaneous eigenspaces. 
Conversely, if $\Gamma$  is a $G$-grading as in \eqref{def_gr}  and we have any character $\chi\in\mathfrak{X}(G)=\mathrm{Hom}(G,\C^\times$), then the map 
$\varphi_\chi\colon\A\to\A $ is an automorphism of $\A$, being $\varphi_\chi(x):=\chi(g)x$ for any $g\in G$ and $x\in\A_g$.  
Moreover $\varphi_\chi$ belongs to the \emph{diagonal group} of $\Gamma$, defined as follows:
\[
\mathrm{Diag}(\Gamma):=\{ \varphi\in\aut(\A): \forall g\in G,\ \exists \alpha_g\in\C^\times\ \text{such that}\ \varphi\vert_{\A_g}=\alpha_g\id\}.
\]
Then $\mathfrak{X}(\mathrm{Diag}(\Gamma))$ is precisely the universal group of $\Gamma$. Furthermore, 
$\Gamma$ is a fine grading if and only if $\mathrm{Diag}(\Gamma)$ is a maximal abelian diagonalizable subgroup (usually called a \emph{MAD-group}) of $\aut(\A)$ (\cite[Theorem 2]{LGI}). For instance, if $S$ is a (complex) semisimple Lie algebra, and $\Gamma: S=\sum_{\alpha\in\Phi\cup\{0\}}S_\alpha$ is the root decomposition relative to a Cartan subalgebra, which is obviously a fine grading, then  $\mathrm{Diag}(\Gamma)=T$ is a maximal torus of the automorphism group, namely, if $\{\a_1,\dots,\a_l\}$ is a set of simple roots of $\Phi$, 
\begin{equation}\label{eq_toro}
{T}=\{t_{s_1, \dots, s_l}:s_i\in\C^\times\}\cong (\C^\times)^l
\end{equation}
where the automorphism $t_{s_1, \dots, s_l}\in\aut(S)$ acts scalarly on the root space $S_\a$  with eigenvalue $s_1^{k_1}\dots s_l^{k_l}$,  for $\alpha=\sum_{i=1}^lk_i\a_i$. Thus the universal group is $\mathfrak{X}(\mathrm{Diag}(\Gamma))\cong \Z^l$.\smallskip

 The fine gradings on the complex algebra $S=\mathfrak e_6$ were classified up to equivalence in \cite{e6}, that is, the MAD-groups of $\aut(\e6)$ were classified up to conjugation. According to this classification, there are 14 non-equivalent fine gradings   on $S$, whose universal grading groups and types are listed in 
  Table~\ref{tablafinasdee6}. Note that these two data together give the equivalence class of the grading. We have also added a column with the interval
 $  \dim_\mathbb{C}S_e\pm\dim_\mathbb{C}\sum_{\stackrel{2g=e}{g\ne e}}S_g$, which will be a tool to apply Proposition~\ref{pr_lemano} lately.

\begin{table}
\begin{center}{
\begin{tabular}{|c|c|c| c|}
\hline\vrule width 0pt height 10pt
 Grading & Universal group & Type & Interval  \cr
\hline\hline\vrule width 0pt height 12pt
 $\Gamma_{1}$ & $ \Z_3^4$   &  $( 72,0,2 )$  &  $0\pm0$ \cr
  \hline \vrule width 0pt height 13 pt
 $\Gamma_{2}$ & $\Z^2\times\Z_3^2$   &  $( 60,9 )$ &  $2\pm0$  \cr
\hline  \vrule width 0pt height 13 pt
$\Gamma_{3}$ & $\Z_3^2\times\Z_2^3$   &  $( 64,7 )$  &  $0\pm14$\cr
\hline \vrule width 0pt height 13 pt
$\Gamma_{4}$ & $\Z^2\times\Z_2^3$   &  $( 48,1,0,7 )$ &  $2\pm28$\cr
\hline \vrule width 0pt height 13 pt
$\Gamma_{5}$ & $\Z^6 $   &  $(72,0,0,0,0,1  )$  &  $6\pm0$\cr
\hline \vrule width 0pt height 13 pt
$\Gamma_{6}$ & $\Z^4\times\Z_2$   &  $(72,1,0,1  )$  &  $4\pm2$\cr
\hline \vrule width 0pt height 13 pt
$\Gamma_{7}$ & $ \Z_2^6$   &  $( 48,1,0,7 )$  &  $0\pm78$\cr   
\hline \vrule width 0pt height 13 pt
$\Gamma_{8}$ & $\Z\times\Z_2^4$   &  $(  57,0,7)$ &  $1\pm29$\cr
\hline \vrule width 0pt height 13 pt
$\Gamma_{9}$ & $\Z_3^3\times\Z_2$   &  $( 26,26 )$ &  $0\pm0$ \cr
\hline \vrule width 0pt height 13 pt
$\Gamma_{10}$ & $\Z^2\times\Z_2^3$   &  $(60,7,0,1   )$ &  $2\pm16$ \cr
\hline \vrule width 0pt height 13 pt
$\Gamma_{11}$ & $\Z_4\times\Z_2^4$   &  $(48,13,0,1     )$ &  $0\pm46$\cr
\hline \vrule width 0pt height 13 pt
$\Gamma_{12}$ & $\Z \times\Z_2^5$   &  $( 73,0,0,0,1  )$ &  $1\pm35$\cr
\hline \vrule width 0pt height 13 pt
$\Gamma_{13}$ & $ \Z_2^7$   &  $( 72,0,0,0,0,1  )$ &  $0\pm78$\cr
\hline \vrule width 0pt height 13 pt
$\Gamma_{14}$ & $ \Z_4^3$   &  $( 48,15 )$ &  $0\pm14$\cr
\hline %
 \end{tabular}\vspace{4pt} %
 \caption{Fine gradings on $\e6$}\label{tablafinasdee6}
 }\end{center}\smallskip
 \end{table} 
 
  \subsection{About real forms}\label{subse_realforms}
  A real Lie algebra $L$ is called \emph{a real form} of a complex Lie algebra $S$ if $L^\C=L\otimes_\R\C=L\oplus\ii L=S$.
  Recall that $L$ is semisimple if its Killing form $\kappa_L$ is non-degenerate. As, for any $x,y\in L$, 
  $\kappa_L(x,y)=\kappa_{L^\C}(x,y)$, consequently the semisimplicity of $L$ equivales to that one of $L^\C$.
  We will abuse of notation calling the \emph{signature} of $L$ (and denoting $\sign(L)$) to the signature of $\kappa_L$, that is, the difference between the number of positive and negative entries in the diagonal of any matrix of the symmetric bilinear form $\kappa_L$ relative to an orthogonal basis.
  For most of the  complex simple Lie algebras $S$, including $S=\e6$, two real forms of $S$ are isomorphic if and only if their signatures coincide, due to the correspondence \eqref{eq_correspondencia} below. 
  There always exists a \emph{compact} real form, with negative definite Killing form, and a \emph{split} real form, with signature equal to the rank of $S$, that is, the dimension of the Cartan subalgebra.
  In case $S=\mathfrak e_6$, there are five real forms up to isomorphism, characterized by their signatures, namely: $-78$, $-26$, $-14$, $2$ and $6$. We will use the notation $\mathfrak{e}_{6,s}$ to refer to any real form of $\e6$ of signature $s$, without specifying a concrete representative of the isomorphism class. So, $\mathfrak{e}_{6,-78}$ is compact and $\mathfrak{e}_{6,6}$ is split.\smallskip

 There is a strong relationship between real forms of a complex Lie algebra $S$
 and involutive automorphisms of $S$.
 If $L$ is  {a real form} of  $S=L\oplus\ii L$, then we can consider the map
 $$
 \sigma\colon S\to S, \quad  \sigma(x+\ii y)=x-\ii y, \quad (x,y\in L)
 $$
  which is a \emph{conjugation} (conjugate-linear order two map) such that $S^\sigma:=\{x\in S:\sigma(x)=x\}$ equals $L$.
   Moreover, two real forms $S^{\sigma_1}$ and $S^{\sigma_2}$ are isomorphic (real) Lie algebras if and only if there is $\varphi\in\aut(S)$ such that
  $\varphi\sigma_1\varphi^{-1}=\sigma_2$. 
  If $\mathfrak{C}_S$ and $\mathfrak{A}_S$ denotes the set of conjugations of $S$ and the set of involutive automorphisms of $S$, respectively, then 
\begin{equation}\label{eq_correspondencia}
  \mathfrak{C}_S/\cong\,\longrightarrow\mathfrak{A}_S/\cong,\qquad  [\si]\mapsto [\theta_\si],
\end{equation} 
 is well defined and bijective, where if $\si$ is any conjugation of $S$, we denote by  $\theta_\si\in\aut(S)$  any involutive automorphism commuting with $\si$ such that the conjugation $\theta_\si\si$ is compact. ($\tau$ is called a \emph{compact} conjugation if the real form $S^\tau$ is compact.) Furthermore, 
\begin{equation}\label{eq_signapartirparte fija}
\sign( {S^\si})=\dim S-2\dim\fix(\theta_\si),
\end{equation}
 where, for $\theta\in\aut(S)$, we denote  the fixed subalgebra by $\fix(\theta)=\{x\in S:\theta(x)=x\}$.
   In particular, Eq.~\eqref{eq_correspondencia} implies that there are 4 order two automorphisms of $\e6$ up to conjugation, characterized by the isomorphy class of their fixed subalgebra.
   This correspondence for  $S=\e6$ is detailed in Table~\ref{tabladeautomorfysignaturas}.

\begin{table}[h]
\begin{tabular}{|c||ccccc|}
\hline  
\vrule width 0pt height 12pt
    $\sign ({S^{\sigma}})$&$-78 $&$2 $&$-14 $&$-26 $&$6 $\\
   \vrule width 0pt height 12pt
   $\dim\fix(\theta_{\sigma})$&$78 $&$ 38$&$ 46$&$ 52$&$36 $\\
 \vrule width 0pt height 12pt
 $ \fix(\theta_{\sigma})$& $\quad E_6\quad  $ & $\quad A_5\oplus A_1\quad $ & $\quad D_5\oplus\C \quad $&$\quad F_4\quad  $&$\quad C_4\quad  $ \\
 \hline  
 \end{tabular}\vspace{3pt}
\caption{Automorphisms versus signatures}\label{tabladeautomorfysignaturas}
 \end{table}

 The order two automorphisms of the complex simple Lie algebras were classified by \'E.~Cartan in \cite{CartanAutomorfismos}.
 This gives the corresponding   real forms by \eqref{eq_correspondencia}. In case $S$ is a classical (complex) simple Lie algebra,
 the possibilities for a real form $L$  are:
   \begin{itemize}
  \item[$\boxed{A_n}$] If $S=\mathfrak{sl}_{n+1}(\C)$,  
  \begin{itemize}
  \item[$\circ$]   $\mathfrak{su}_{p,q}=\{x\in\mathfrak{sl}_{n+1}(\C):I_{p,q}x+\bar x^t I_{p,q}=0\}$, with $p+q=n+1$, of signature $-(n^2+2n)+4pq$;
  \item[$\circ$] $\mathfrak{sl}_{n+1}(\R)$, of signature $n$; 
  \item[$\circ$] $\mathfrak{sl}_m(\H)=\{x\in\mathfrak{gl}_{m}(\mathbb H):\textrm{Re}(\tr(x))\}=0$,  with odd $n=2m-1>1$, of signature $-2m-1=-n-2$.
  \end{itemize}
   \item[$\boxed{B_n, D_n}$] If $S\in\{\mathfrak{so}_{2n+1}(\C),\mathfrak{so}_{2n}(\C)\}$,
   \begin{itemize}
    \item[$\circ$]  $\mathfrak{so}_{p,q}(\R)=\{x\in\mathfrak{gl}_{p+q}(\mathbb R):I_{p,q}x+ x^t I_{p,q}=0\}$, with either  $p+q=2n+1$ (type $B_n$) or $p+q=2n$ (type $D_n$), of signature $\frac{(p+q)-(p-q)^2}2$;

  \item[$\circ$]  $\mathfrak{so}^*_{2n}(\R)=\{x\in\mathfrak{gl}_{n}(\mathbb H):x^th+ h\bar x=0\}$, where $h=\diag(\ii,\dots,\ii)$, of signature $-n$.
  This case only happens when $n>4$ ($D_n$).
    \end{itemize}
    \item[$\boxed{C_n}$] If $S=\mathfrak{sp}_{2n}(\C)$, 
    \begin{itemize} 
   \item[$\circ$]   $\mathfrak{sp}_{2n}(\R)=\{x\in\mathfrak{gl}_{2n}(\mathbb R): x+\bar x^t=0\}$, of signature $n$; 
    \item[$\circ$]    $\mathfrak{sp}_{p,q}(\H)=\{x\in\mathfrak{gl}_{n}(\mathbb H):I_{p,q}x+\bar x^t I_{p,q}=0\}$, with $p+q=n$, of signature $-2(p-q)^2-(p+q)$.
  \end{itemize}
    \end{itemize}
  Here $I_{p,q}=\diag(1,\dots,1,-1,\dots,-1)$ (the first $p$ entries in the diagonal are $1$  and the $q$ remaining entries are $-1$) and $I_n\equiv I_{n,0}$ is so the identity matrix. Also we denote by $\mathfrak{so}_{n}(\R)\equiv \mathfrak{so}_{n,0}(\R)$, and similar conventions are used for $\mathfrak{sp}$ and $\mathfrak{su}$, and for $\C$ and $\H$. Each real form will also be   denoted according to  its signature and the type of the  complexification, for instance, 
  $\mathfrak{sl}_m(\H)\equiv \mathfrak{a}_{2m-1,-2m-1}$. (In particular, $\mathfrak{a}_{2m-1,-2m-1}$ denotes also any algebra in the isomorphy class of $\mathfrak{sl}_m(\H)$.)
  Throughout  the text, $E_{ij}$ will denote a  matrix, of size depending on the context, where the only non-zero entry will be the $(i,j)$th, equal to $1$.
So   $\diag(s_1,\dots,s_n)$ means $\sum_{i=1}^ns_iE_{ii}$.    
\smallskip

 
 Note that the signatures of the subalgebras of a Lie algebra are related, in a certain way, to the signatures of the whole Lie algebra, although the Killing form of the subalgebra does not coincide with the restriction of the Killing form of the algebra. To this aim, we state the next result for complex Lie algebras.
  
      \begin{lm}\label{lema_positivo}
      Let $S_0$ be a simple subalgebra of a complex Lie algebra $S$. 
    Then  there is a positive rational number $ r\in\mathbb Q_{>0}$ such that $\kappa_S\vert_{S_0}=r\,\kappa_{S_0}$. 
            \end{lm}   
   
   \begin{proof}
   If $\mathfrak{g}$ is a simple Lie algebra and $f_1,f_2\colon  \mathfrak{g}\times\mathfrak{g}\to\C$ are non-zero $\mathfrak{g}$-invariant maps, then $\tilde{f_i}\colon \mathfrak{g}\to\mathfrak{g}^*,x\mapsto f_i(x,-) $ are isomorphisms of $\mathfrak{g}$-modules and hence $\phi=\tilde{f_2}^{-1}\tilde{f_1}\in \Hom_\mathfrak{g}(\mathfrak{g},\mathfrak{g})$. Now $\phi\in \C\id_\mathfrak{g}$ (Schur's lemma) because any  eigenspace must be a non-zero $\mathfrak{g}$-submodule of the adjoint module $\mathfrak{g}$,   which is irreducible, so the eigenspace is the whole $\mathfrak{g}$.  If we apply this fact to $\mathfrak{g}=S_0$, $f_1=\kappa_S\vert_{S_0}$ and $f_2=\kappa_{S_0}$, we get that 
   there  exists $  r\in\C$ such that 
$\kappa_S\vert_{S_{ 0}}=r\,\kappa_{S_{ 0}}$.  

Let us denote by $\Phi$ a root system of $S_0$ relative to a Cartan subalgebra $\mathfrak{h}$ and $\{\alpha_i\}_{i=1}^l$ a set of simple roots of $\Phi$. Thus $\mathfrak{h}$ is spanned by $\{h_{\alpha_i}\}_{i=1}^l$, which satisfy $\lambda_i(h_{\alpha_j})=\delta_{ij}$ for $\lambda_i$ the fundamental dominant weights. As $h=h_{\alpha_1}\in S_0$, in particular $\kappa_S\vert_{S_{0}}(h,h)=r\,\kappa_{S_{0}}(h,h)$. The last Killing form is easy to compute since $\ad h$ acts scalarly on each root space $(S_0)_\a$ with scalar $\a(h)=\langle\a,\a_1\rangle\in\Z$, so that 
$\kappa_{S_{0}}(h,h)=\sum_{\alpha\in\Phi}\alpha^2(h)$  is a non-negative integer. In fact, it is positive, since $\alpha_1(h)=2$. 
Furthermore, for any representation $\rho\colon S_0\to\mathfrak{gl}(V)$,   
the endomorphisms $\rho(\mathfrak{h})$ are simultaneously diagonalizable, and the simultaneous eigenspaces are the weight spaces
\cite{Draper:Humphreysalg}. If we denote by $\Lambda(S )\subset\mathfrak{h}^*$ the weights of the $S_0$-module $S $, then $\Lambda(S)\subset\sum_{i=1}^l\Z \lambda_i$, so that 
$\kappa_{S}(h,h)= \sum_{\mu\in\Lambda(S )}\mu^2(h) \in\mathbb Z_{\ge0}$ and hence $r\in\mathbb Q_{\ge0}$.
Besides, as $\Phi\subset \Lambda(S)$, then $\kappa_{S}(h,h)=\kappa_{S_{0}}(h,h)+ \sum_{\mu\in\Lambda(S )\setminus\Phi}\mu^2(h)$ is strictly positive, and so $r$ is.
      \end{proof}

      This gives immediately the required facts for signatures of simple subalgebras.
      
      \begin{lm}\label{lema_kil}
      Let $L_0$ be a simple subalgebra of a real Lie algebra $L$. 
    Then  there is a positive rational number $ r\in\mathbb Q_{>0}$ such that $\kappa_L\vert_{L_0}=r\,\kappa_{L_0}$. 
    In particular, $\sign(\kappa_L\vert_{L_0})=\sign(L_0)$.
            \end{lm}   
            
            \begin{proof} 
            The first result is clear from Lemma~\ref{lema_positivo}, since $\kappa_L\vert_{L_0}=\kappa_{L^\C}\vert_{L_0}$ and also $\kappa_{L_0}=\kappa_{L_0^\C}\vert_{L_0}$. For the claim relative to the signatures, observe that the key point is that the scalar $r$ is positive.
       \end{proof}

\subsection{About gradings on real forms}
 
If $L$ is a real form of a complex Lie algebra $S$, there is a close relationship between   gradings on $L$ and on $S$.

  If $\Gamma: L = \oplus_{g\in G} L_g$ is a grading on a real Lie algebra $L$, let us denote by $\Gamma^\C$ the grading on $L^\C$ given by   $\Gamma^{\C}: L^{\C}=  \oplus_{g\in G} (L_g)^\C=\oplus_{g\in G} (L_g \oplus \ii L_g)$. 
  It will be called the  \emph{complexified grading}.
  It is important to observe that, if $\Gamma^{\C}$ is a fine grading on $L^\C$, then $\Gamma$ is a fine grading on $L$.

Let $\Gamma_S: S = \oplus_{g\in G} S_g$ be a grading on a complex Lie algebra $S$ and let $L$ be a real form of $S=L^\C$. We will say that $L$ \emph{inherits} the grading $\Gamma_S$ if there exists a grading  $\Gamma$ on $L$ such that $\Gamma^\C=\Gamma_S$, that is, $L=\oplus_{g\in G} (L\cap S_g)$.

By abuse of notation, we will say that $\sig{14}$ inherits the grading $\Gamma_i$ in Table~\ref{tablafinasdee6} ($i=1,\dots,14$) if some real form of $\e6$ of signature $-14$ has a grading whose complexified grading is  {equivalent} to $\Gamma_i$.\smallskip

A useful criterion to check if a determined real form of a complex Lie algebra inherits a grading, assumed we know  another real form   inheriting it,   can be stated as a direct consequence of \cite[Proposition~3]{reales}:
  
 \begin{pr}\label{pr_nuestrometodo} 
 Let $\si$ be a conjugation of $S$ such that $S^\si$ inherits certain fine grading $\Gamma: S = \oplus_{g\in G} S_g$.
 For any other conjugation $\mu$ of $S$, the real form $S^\mu$ inherits $\Gamma$ if and only if $\mu\si^{-1}\in \mathrm{Diag}(\Gamma)\le\aut(S)$. 
 \end{pr}

Gradings and Killing forms are closely related. For instance, if $L = \oplus_{g\in G} L_g$ is a grading over an abelian group $G$, then two homogeneous components $L_g$ and $L_h$ are orthogonal for the Killing form unless $g+h$ is the neutral element $e\in G$. In particular, the homogeneous components corresponding to degrees of order not 2 are totally isotropic subspaces.\smallskip

               \begin{lm}\label{lema_parejas}
  For $L=L_{\bar0}\oplus L_{\bar1}$   a real $\Z_2$-graded algebra   and $0\ne t\in\R$, denote by $L^t:=(L,[\ ,\ ]^t)$ the Lie algebra with the same underlying vector space but new product given by
     \begin{equation}\label{eq_Lapareada}
     [x_0+v_0,x_1+v_1]^t=[x_0,x_1]+[x_0,v_1]+[v_0,x_1]+t[v_0,v_1],
     \end{equation}
   if $x_i\in L_{\bar0}$ and $v_i\in L_{\bar1}$. Then, for $\varepsilon(t)=t/|t|\in\{\pm1\}$: 
   \begin{itemize}
   \item[a)] $\sign(L^t) = \sign(\kappa_L\vert_{L_{\bar0}})+\varepsilon(t)\sign(\kappa_L\vert_{L_{\bar1}})$,
   \item[b)] $\sign(L)+\sign(L^{-1})=2\sign(\kappa_L\vert_{L_{\bar0}})$,
   \item[c)] $\sign(L)+\sign(L^{-1})= 2\sign(L_{\bar0})$ in case  $L_{\bar0}$ is simple.
   \end{itemize}
   \end{lm}   
   
   \begin{proof} For item a), note that $\kappa_{L^t}(x_0,x_1)=\kappa_{L}(x_0,x_1)$; $\kappa_{L^t}(x_0,v_1)=0=\kappa_{L}(x_0,v_1)$
   and $\kappa_{L^t}(v_0,v_1)=t\,\kappa_{L}(v_0,v_1)$. 
   Item b) is obtained when adding the  two equations obtained in item a) for $t=1$ and  $t=-1$. 
   Item c) is consequence of item b) and  Lemma~\ref{lema_kil}:
   $\sign(\kappa_L\vert_{L_{\bar0}})=\sign(L_{\bar0})$. 
   \end{proof} 
   
              This result relates the signature of a $\Z_2$-graded Lie algebra $L=L_{\bar0}\oplus L_{\bar1}$ with that one of $ L_{\bar0}\oplus\ii L_{\bar1}=L^{-1}$, which is a different real form of $L^\C$. This is important because any grading on $L$ which is a  refinement of such $\Z_2$-grading is immediately a grading of the other real form $L^{-1}$ (and conversely).\smallskip

       Finally, we put our attention on the next result,  extracted from \cite[Proposition~9]{sig26}, which bounds the admissible signatures of a real form inheriting a determined grading.
  
  \begin{pr}\label{pr_lemano}
Let $S$ be a complex simple Lie algebra and $L$ a real form of $S$. Suppose that $L$ inherits    a fine grading on $S$ given by $ S=\sum_{g\in G}S_g$. Then
$$
|\sign (L) - \dim  S_e | \leq \sum_{\stackrel{e\ne g\in G}{2g=e}}\dim  S_g.
$$
\end{pr}
This is the reason why we have added,  in the last column in  Table~\ref{tablafinasdee6},  the possible interval for the signature  $  \dim_\mathbb{C}S_e\pm\dim_\mathbb{C}\sum_{\stackrel{2g=e}{g\ne e}}S_g$    for each grading $\Gamma_i$.

\subsection{About related structures}

Throughout this paper we will use   several non-associative algebras.  

First, we denote by $\OO$ the \emph{octonion algebra}, which is the real division algebra endowed with a norm $n$ such that $\{1, e_i:1\le i\le 7\}$ is an orthogonal basis and 
the product is given by recalling Fano plane in Figure~\ref{fig_fano}:   each triplet $ \{e_{ i},e_{j },e_{k }\}$ in one of the seven \emph{lines} (the circle is also a line) spans a subalgebra isomorphic to $\R^3$ with the usual cross product. (Thus  $ \R\langle1,e_{ i},e_{j },e_{k }\rangle$ is isomorphic to the   quaternion algebra $\mathbb H$ when $e_{ i}$, $e_{j }$ and $e_{k }
$ are the three different elements in one line. That is, $e_ie_je_k=-1$ in the sense of the arrows.) Recall that 
$\overline{a}:=r1-\sum s_ie_i$ if $a=r1+\sum s_ie_i$, so that the norm and the trace are given by $n(a)=a\overline{a}$ and $t_\OO(a)=a+\overline{a}$. The subspace of zero trace octonions (or \emph{imaginary} octonions) is denoted by $\OO_0$.

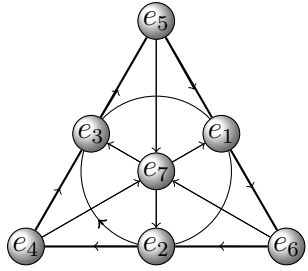
\begin{figure}
 \begin{center}  
    \begin{tikzpicture}
\tikzstyle{point}=[ball color=white, circle, draw=black, inner sep=0.02cm]
\node (v7) at (0,0) [point] {$e_7$};
\draw[->] (0,0) circle (1cm);
 \draw[thick,-<] ({-1/sqrt(2)},{-1/sqrt(2)}) -- ({-1/sqrt(2)+0.001},{-1/sqrt(2)-0.001});
\node (v1) at (90:2cm) [point] {$e_5$};
\node (v2) at (210:2cm) [point] {$e_4$}; 
\node (v4) at (330:2cm) [point] {$e_6$};
\node (v3) at (150:1cm) [point] {$e_3$};
\node (v6) at (270:1cm) [point] {$e_2$};
\node (v5) at (30:1cm) [point] {$e_1$};  
\draw [->, color=black] (v4) -- (v7);
\draw [->, color=black] (v1) -- (v7);
\draw [->, color=black] (v2) -- (v7);
\draw [->, color=black] (v7) -- (v3);
\draw [->, color=black] (v7) -- (v6);
\draw [->, color=black] (v7) -- (v5);
\draw [->, color=black] (v6) -- (230:13mm);\draw [->, color=black] (v4) -- (310:13mm); 
\draw [->, color=black] (v1) -- (65:12mm);\draw [->, color=black]  (v3) --  (115:12mm) ; 
\draw [->, color=black] (v5) -- (350:13mm);\draw [->, color=black] (v2) -- (190:13mm);
\draw[-,thick] (v1) -- (v3) -- (v2);
\draw[-,thick] (v2) -- (v6) -- (v4);
\draw[-,thick] (v4) -- (v5) -- (v1);
\draw (v3) -- (v7) -- (v4);
\draw (v5) -- (v7) -- (v2);
\draw (v6) -- (v7) -- (v1);
\end{tikzpicture}
\caption{Fano plane}\label{fig_fano}
\end{center}
\end{figure}

Second,   an $\mathbb F$-algebra is called a \emph{Jordan algebra} if it is   commutative  and it satisfies the Jordan identity
$
(x^2\cdot y)\cdot x=x^2\cdot (y\cdot x)
$.
The examples of real Jordan algebras more relevant for our purposes are
$$
 \J_c=\mathcal{H}_3(\OO,\gamma_1),\quad
 \J=\mathcal{H}_3(\OO,\gamma_2), \quad  
\M=\mathcal{H}_3(\C,\gamma_3),\quad
\M_s=\mathcal{H}_3(\R\oplus\R,\gamma_1);
 $$
where if $\mathcal C$ is either  $\OO$, or $\C$, or $\R\oplus\R$ (with the exchange involution $\overline{(a,b)}=(b,a)$), 
\begin{equation*}\label{eq_jordanalg}   
\mathcal{H}_3(\mathcal C,\gamma):=\{x \in \textrm{Mat}_{3\times 3}(\mathcal C): \gamma\bar x^t\gamma^{-1}=  x \},
\end{equation*}
the so called \emph{symmetrized} product is given by
\begin{equation}\label{eq_sympr}
x\cdot y:=\frac12(xy+yx),
\end{equation}
and $\gamma$ is either  $\gamma_1=I_3$ or $\gamma_2=\diag(1,-1,1)$ or $\gamma_3=E_{11}+E_{23}+E_{32}$. (Note that $\M=\mathcal{H}_3(\C,\gamma_3)$ is of course isomorphic to the Jordan algebra $ \mathcal{H}_3(\C,\gamma_2)$, but our preference is due to the fact the gradings in Eq.~\eqref{eq_Z32} will be more easily formulated in terms of $\M$.)
 
 If $A$ is an associative algebra, $A^+$ denotes the Jordan algebra obtained when considering the symmetrized product \eqref{eq_sympr} on the vector space $A$. Observe that $\M_s$ is isomorphic to $\textrm{Mat}_{3\times 3}(\R)^+$.   \smallskip

 \section{From the Albert algebra to gradings over $\Z_2^6$ and  $\Z\times\Z_2^4$}\label{deAlbert}
 \subsection{The model}
  
Nathan Jacobson described all the real forms of $\e6$ in his book about exceptional Lie algebras  \cite[Eq.~(147)]{Jacobsondeexcepcionales}. The five of them can be
 obtained as  $L =(\der (J)\oplus   J_0)^\pm$ for $J$ a real form of the Albert algebra $\J_c^\C$, which is the only exceptional complex Jordan algebra, of dimension 27. The product on $L $ is given  by
 \begin{itemize}
 \item The natural action of $\der (J)$ on $J_0$;
 \item If $x,y\in J_0$, $[x,y]:=\pm [R_x,R_y]\in\der (J)$,
 where
 $R_x\colon J\to J$ denotes the operator multiplication $R_x(y)=x\cdot y$.
  \end{itemize}
  Thus $L $ becomes a $\Z_2$-graded Lie algebra in both cases $\pm$, with even part $L_{\bar0}=\der (J)$ and odd part $L_{\bar1}=J_0$.  
  A first model of $\sig{14}$
is 
$$
\L=(\der(\J)\oplus \J_0)^-, \text{ for } \J=\mathcal{H}_3(\OO,\diag\{1,-1,1\}).
$$ 
 Indeed,  
  the even and the odd part of any $\Z_2$-graded Lie algebra are orthogonal for the Killing form $\kappa$ of $\L$, so   
    $\sign (\L)=\sign \kappa\vert_{L_{\bar0}}+ \sign \kappa\vert_{L_{\bar1}}$.
On one hand,  $\der(\J)\cong\mathfrak{f}_{4,-20}$, and, by Lemma~\ref{lema_kil}, 
the identity $\sign(\kappa\vert_{L_{\bar0}})=\sign({L_{\bar0}})=-20$ holds.
 On the other hand, the traceform of the   Jordan algebra $\J$ ($(x,y)\mapsto\tr(x\cdot y)$)   has signature equal to $-6$ \cite[p.~114]{Jacobsondeexcepcionales}, which coincides with $-\sign \kappa\vert_{L_{\bar1}}$, due to choice of sign when multiplying two odd elements (see also Lemma~\ref{lema_parejas}a).

  
  \subsection{The gradings}

  All the $G$-gradings on $\J$ induce naturally $G$-gradings on $ \der(\J)$ and $G\times\Z_2$-gradings on $\L$. The Jordan algebra $\J$ is $\Z_2^5$-graded and $\Z\times\Z_2^3$-graded \cite[Theorem~7 and Corollary~1(3)]{reales}. Furthermore, the complexifications of the $ \Z_2^6$-grading and of the $\Z \times\Z_2^4$-grading induced on $\L$ are just  $\Gamma_7$ and $\Gamma_8$, respectively.  In particular, the real form $\sig{14}$ admits a fine $ \Z_2^6$-grading as well as a fine $\Z \times\Z_2^4$-grading.\smallskip
 
 The  fine gradings  on $\J$ were recalled in \cite{sig26} since   $(\der\J\oplus \J_0)^+$ has just signature $-26$, and hence both 
 $\sig{26}$ and $\sig{14}$  share the gradings coming from gradings on  the Jordan algebra $\J$. For completeness we   enclose a description here.

 The octonion algebra $\OO$ is   $\Z_2^3$-graded:
  \begin{equation}\label{eq_gradenO} 
    \OO_{(\bar1,\bar0,\bar0)}=\R e_1, \quad  \OO_{(\bar0,\bar1,\bar0)}=\R e_2, \quad \OO_{(\bar0,\bar0,\bar1)}=\R e_7.
 \end{equation}
This induces a $\Z_2^3$-grading on the Jordan algebra $\J$ by means of
\begin{equation}\label{eq_gradingdeoctoniones}
\begin{array}{l}
\J_{(\bar0,\bar0,\bar0)}=\sum_i\R E_{ii}\oplus\sum_i\R\iota_i(1),\\
\J_g=\sum_i\iota_i(\OO_g) \quad \textrm{ if }e\ne g\in\Z_2^3,
\end{array}
\end{equation}
where an arbitrary element  in $\J$ is denoted  by
 $$
 \left(\begin{array}{ccc}
 s_1&a_3&\bar a_2\\-\bar a_3&s_2&a_1\\a_2&-\bar a_1&s_3\end{array}\right)=:\sum_i s_iE_{ii}+\sum_i\iota_i(a_i),
 $$
 if $s_i\in\R$ and $a_i\in\OO$. 
 The grading \eqref{eq_gradingdeoctoniones} is compatible with the    $\Z_2^2$-grading on $\J$  defined by
$$
\begin{array}{ll}
\J_{(\bar0,\bar0)}=\sum_i\R E_{ii}, \quad &\J_{(\bar0,\bar1)}=\iota_1(\OO),\\
\J_{(\bar1,\bar0)}=\iota_2(\OO),&\J_{(\bar1,\bar1)}=\iota_3(\OO),
\end{array}
$$
getting then the mentioned $\Z_2^5$-grading on $\J$. 

Also, the $\Z\times\Z_2^3$-grading on $\J$ is obtained by refining  the $\Z_2^3$-grading \eqref{eq_gradingdeoctoniones} with the next   $\Z$-grading on $\J$:
$$
\begin{array}{l}{\J }_{-2}=\R({ E_{22}-E_{33}-\iota_1(1)}),\\
\J_{-1}=\{{ \iota_2(a)-\iota_3(\overline{a}):
a\in \OO}\},\\
{\J}_{0}= \span{\{E_{11}, E_{22}+E_{33},
\iota_1(a): a\in \OO_0\}},\\
{\J}_{1}=\{{
\iota_2(a)+\iota_3(\overline{a}): a\in \OO}\},\\
{\J}_{2}=\R({  E_{22}-E_{33}+\iota_1(1)}),
\end{array}
$$ 
which is just the eigenspace decomposition relative to  the derivation
$4[R_{\iota_1(1)},R_{E_{22}}]\in\der(\J)$.

 \section{A $\Z_2^3\times\Z_3^2$-grading based on the Tits construction}\label{deTits}
 \subsection{The model}
 In 1966 \cite{Tits}, Tits provided a beautiful unified construction of all the exceptional simple Lie algebras. We review here only some particular case of  Tits' construction to  
  get a (not usual) model of $\sig{14}$.

  Consider, for the Jordan algebra $\M =\mathcal{H}_3(\C,\gamma_3)$, the vector space
\begin{equation}\label{eq_TitsModel}
\T(\OO,\M):=\der(\OO)\oplus (\OO_0 \otimes \M_0) \oplus \der(\M),
\end{equation}
which is made into a Lie algebra over $\R$ by defining the multiplication   
(bilinear and anticommutative)
which agrees with the ordinary commutator on the Lie algebras $\der(\OO)$ and $\der(\M)$ and it satisfies
\begin{equation}\label{eq_TitsProduct}
\begin{array}{l}
\bullet\  {[}\der(\OO),   \der(\M)]=0, \\
\bullet\  {[}d, a\otimes x]=d(a) \otimes x, \\
\bullet\  [D, a\otimes x]=a \otimes D(x), \\
\bullet\  {[}a\otimes x, b\otimes y]= \frac13\tr(x\cdot y) d_{a,b}+[a,b]\otimes (x\ast y)+ 2t_\OO(ab)[R_x,R_y],
\end{array}
\end{equation}
for all $d\in \der(\OO)$, $D\in \der(\M)$, $a,b\in \OO_0$ and $x,y\in \M_0$. 
The used notations are, for   $a,b,c\in\OO$ and $x,y\in\M$,
\begin{itemize}
\item[$\circ$]
$[a,b]=ab-ba\in\OO_0$, 
\item[$\circ$]
 $
d_{a,b}:=[l_a,l_b]+[l_a,r_b]+[r_a,r_b]\in\der(\OO),
$
for $l_a(b)=ab$ and $r_a(b)=ba$  the left and right multiplication operators in $\OO$ respectively, so that 
$d_{a,b}(c)=[[a,b],c]+3(ac)b-3a(cb)$; 
\item[$\circ$]
$x*y:=x\cdot y-\frac13\tr(x\cdot y)I_3\in \M_0$, 
\end{itemize}
and again $R_x$ is the multiplication operator in the Jordan algebra and $[R_x,R_y]$ the commutator.

 \begin{pr}\label{pr_constTitsnuestra}
$\T(\OO,\M)\cong\sig{14}$ 
\end{pr}

\begin{proof}
 It is well-known that $\T(\OO^\C,\M^\C)$ is a complex Lie algebra of type $\mathfrak e_6$, so that we only need to compute the signature of the Killing form $\kappa$ of the real form $\T(\OO,\M)$. As far as we know, it can not be found in the literature. 
 For the computation, we follow the lines  of the proof of \cite[Proposition~2]{sig26}, although the ideas are mainly based in  Jacobson's book \cite{Jacobsondeexcepcionales}. Observe the following facts:
 \begin{itemize}
 \item[a)]  $\der(\OO)$, $\OO_0 \otimes \M_0$ and $\der(\M)$ are three orthogonal subspaces for the Killing form (\cite[p.~116]{Jacobsondeexcepcionales}).
 \item[b)] If $d,d'\in \der(\OO)$, then $ \kappa(d,d')=12\tr(dd')=3\kappa_{\g2}(d,d')$, denoting by $\kappa_{\g2}$ the Killing form of the algebra $\der(\OO)=\frak{g}_{2,-14}$.
 This implies that the signature of $\kappa\vert_{\der(\OO)}$ is the same as the one of $\kappa_{\g2}$, that is, -14.
 \item[c)] If $D,D'\in \der(\M)$, then $ \kappa(D,D')=8\tr(DD')=8\kappa_{\frak{a}_2}(D,D')$, denoting by $\kappa_{\frak{a}_2}$ the Killing form of the algebra $ \der(\M)\cong\mathfrak{su}_{1,1}=\frak{a}_{2,0}$.
 This implies that the signature of $\kappa\vert_{\der(\M)}$ is the same   as the one of $\kappa_{\frak{a}_2}$, that is, $0$.
 \item[d)] For each $a,b\in\OO_0$ and $x,y\in\M_0$, we compute $ \kappa(a\otimes x,b\otimes y)=-60\, n(a,b)\tr(x\cdot y)$.
 As $n$  is positive definite, the signature of $\kappa\vert_{\OO_0 \otimes \M_0}$  will coincide with minus 7 times the signature of the traceform of $\M_0$ (the bilinear form $(x,y)\mapsto\tr(x\cdot y)$), which is equal to 0.  
 \end{itemize}
  Consequently, the signature of $\T(\OO,\M)$ turns out to be
 $-14$.
 \end{proof}
 
It can be observed that  the facts about signatures in items b) and c) are also consequences of Lemma~\ref{lema_kil}, and that it was no necessary   to compute the signature of the traceform in d) since -14 is the only signature of a real form of $\e6$ which is multiple of 7.

 \subsection{The grading}
 A $\Z_2^3\times\Z_3^2$-grading is achieved when considering the   $\Z_2^3$-grading on $\OO$  given by Eq.~\eqref{eq_gradenO} and the following fine $\Z_3^2$-grading on $\M$: 
\begin{equation}\label{eq_Z32}
 \deg  \begin{pmatrix}
 0&0 & 1  \cr
 1&0&0 \cr
 0&1&0
 \end{pmatrix} =(\bar1,\bar0);\quad \deg  \begin{pmatrix}
 1&0 & 0  \cr
 0&\omega&0 \cr
 0&0&\omega^2
 \end{pmatrix} =(\bar0,\bar1),
 \end{equation}
 since $\M$ is generated (as an algebra) by these two matrices. 
 There is  a multiplicative basis of $\M$ (the product of two elements in such a  basis is multiple of a third one) formed by homogeneous elements of the 
   $\Z_3^2$-grading, in fact, all of them are invertible matrices.
 In particular $\M_0$ breaks into eight one-dimensional homogeneous components. 
 The obtained $\Z_2^3\times\Z_3^2$-grading on $\T(\OO,\M)\cong\sig{14}$ is fine, since its complexified grading is just $\Gamma_3$.

\begin{re} \rm{Tits' construction can be applied by replacing in $\T(\OO,\M)$   the Jordan algebra $\M$ with $\mathcal{H}_3(\mathcal C,\gamma_i)$, for $\mathcal C\in\{\C,\R\oplus\R\}$, or the octonion algebra $\OO$ with the split octonion algebra $\OO_s$, getting in this way all the five real forms of $\e6$.
 The same arguments as in Proposition~\ref{pr_constTitsnuestra} can be used to conclude  that $\T(\OO_s,\M)\cong\mathfrak{e}_{6,2}$, which implies that 
$\mathfrak{e}_{6,2}$ possesses also a fine $\Z_2^3\times\Z_3^2$-grading whose complexified grading is $\Gamma_3$.}\end{re}

 \section{Flag model and related gradings over $\Z\times\Z_2^5$ and $\Z^2\times\Z_2^3$}\label{flag}
 
 Recall   that $\Gamma_{12}$ and $\Gamma_{10}$, the fine $\Z\times\Z_2^5$ and $\Z^2\times\Z_2^3$-gradings on $S=\e6$, 
 are both refinements of the   $\Z$-grading on $S$ produced by the one-dimensional torus  $\{t_{1,s,1,1,1,1}:s\in\C^\times\} $ contained in the maximal torus $T$ considered in \eqref{eq_toro}.
 In other words, if $S=\mathfrak{h}\oplus(\oplus_{\alpha\in\Phi }S_\alpha)$ is the root decomposition relative to a Cartan subalgebra $\mathfrak{h}$, and $\{\alpha_i\}_{i=1}^6$ is a set of simple roots of the root system $\Phi$  (ordered as in Figure~\ref{fig_Satake}),
 the $\Z$-grading is $S=\oplus_{n=-2}^2S_n$, for
 \begin{equation}\label{eq_lasgrads}
 S_n=\{\oplus_{\alpha\in\Phi}S_\alpha:\alpha=\sum_i k_i\alpha_i,k_2=n\}.
 \end{equation}
 By counting roots, we see that 
 $S_0=\mathfrak{h}\oplus\{\oplus_{\alpha\in\Phi}S_\alpha:\alpha=\sum_i k_i\alpha_i,k_2=0\}$ has dimension $36$, 
 $\dim S_{\pm1}=20$ and $\dim S_{\pm2}=1$. Also, the Dynkin diagram of the semisimple part of the homogeneous component $S_0$ is obtained by removing  the node $\alpha_2$ of  the Dynkin diagram, so $S_0$ has type $A_5\oplus\C$ ($\C$ denoting here a one-dimensional centre). 
 A linear model useful for our purposes is developed in \cite[\S3.4]{Weyle6}. Take $V$ a $6$-dimensional vector space over $\C$, 
 and then
 \begin{equation}\label{eq_la5grad}
S_V:=\wedge^6V^*\oplus\wedge^3V^*\oplus\mathfrak{gl}(V)\oplus\wedge^3V\oplus\wedge^6V
 \end{equation}
can be endowed with a Lie algebra structure such that $S_V\cong\e6$. The products can be  described with multilinear algebra: in terms of contractions, of the wedge products, and   also making use of the dualization of the product $\cdot\colon S_n\times S_{-n}\to\C$, which gives the bracket $[S_n, S_{-n}]\subset S_0$. Precisely, if $u\in S_n$, $f\in S_{-n}$, then $[u,f]$ denotes the only element in $\mathfrak{gl}(V)$ characterized by 
\begin{equation}\label{eq_producots}
\tr(g\circ[u,f])=g(u)\cdot f\in\C
\end{equation}
 for all $g\in \mathfrak{gl}(V).$

 The first question, then, is if $\sig{14}$ inherits this $\Z$-grading. 
 This can be answered affirmatively as a consequence of the following result due to Cheng \cite[Theorem~3]{Cheng}.

\begin{pr}
A simple real Lie algebra $L$ admits a grading $L=L_{-2}\oplus L_{-1}\oplus L_{0}\oplus L_{1}\oplus L_{2}$ such that $\dim L_2=1$ if and only if there is a   long root corresponding to a white node such that its restricted multiplicity is equal to 1.
\end{pr}

This can be applied to $L=\sig{14}$ as detailed in \cite[17.6]{Satakes}. Precisely,
  the Satake diagram is the following
 \begin{center}{
\begin{picture}(23,5)(4,-0.5)
\put(5,0){\circle{1}} \put(9,0){\circle*{1}} \put(13,0){\circle*{1}}
\put(17,0){\circle*{1}} \put(21,0){\circle{1}}
\put(13,4){\circle{1}}
\put(5.5,0){\line(1,0){3}}
\put(9.5,0){\line(1,0){3}} \put(13.5,0){\line(1,0){3}}  
\put(17.5,0){\line(1,0){3}}
\put(13,0.5){\line(0,1){3}}
\put(4.5,1){$\scriptstyle \alpha_1$} \put(8.5,1){$\scriptstyle \alpha_3$}
\put(12.85,1){$\scriptstyle \alpha_4$} \put(16.5,1){$\scriptstyle \alpha_5$}
\put(20.5,1){$\scriptstyle \alpha_6$} \put(13.9,3.6){$\scriptstyle \alpha_2$}
\cbezier (6,-1)(11,-2.5)(15,-2.5)(20,-1) 
 \put(20,-1){\vector(2,1){1}}
 \put(6,-1){\vector(-2,1){1}}
 \end{picture}
 \begin{figure}[!h]
\caption{Satake diagram of $\sig{14}$}\label{fig_Satake}
\end{figure}
}\end{center}\vspace{-0.3cm}
The restricted root system is of type $BC_2$, namely,
$$
\pm\{\b_1,\b_2,\b_1+\b_2,2\b_1+\b_2,2\b_1,2\b_1+2\b_2\}\subset\mathfrak{a}^*,
$$
for $\mathfrak{a}$ a maximal abelian subalgebra of the eigenspace of a Cartan involution relative to the eigenvalue $-1$. 
There is a (long) root, the maximal one 
$\tilde\a=\alpha_1+2\alpha_2+2\alpha_3+3\alpha_4+2\alpha_5+\alpha_6$,
such that the restricted root $\tilde\a\vert_{\mathfrak a}=2\b_1\ne0$ has restricted multiplicity equal to 1
(in other words, $\{\a\in\Phi:\a\vert_{\mathfrak a}=\tilde \a\vert_{\mathfrak a}\}=\{\tilde\a\}$).
So, the diagonalization of any non-zero element   in the corresponding root space   produces the desired $\Z$-grading. 

\begin{re}\label{re_solouncontacto}\rm{
Note that the only $\Z$-grading $S=\oplus_{n=-2}^2 S_n$ of the complex Lie algebra $\e6$ such that $\dim S_2=1$ is that one in \eqref{eq_lasgrads}, up to \emph{isomorphism}. This type of gradings  usually receives the name of \emph{contact} gradings. We can argue as follows. By \cite[Chapter~3]{enci}, any $\Z$-grading on $S$ comes from a choice of non-negative integers $(l_1,\dots,l_6)$ in  such a way that  (for $n\ne0$)
\begin{equation*} 
 S_n=\{\oplus_{\alpha\in\Phi}S_\alpha:\alpha=\sum_i k_i\alpha_i,n=\sum_i k_i l_i\}.
 \end{equation*}
 For instance, the grading \eqref{eq_lasgrads} corresponds to the choice $(0,1,0,0,0,0)$. For a grading with just 5 pieces, as the maximal root is $\tilde\alpha=\alpha_1+2\alpha_2+2\alpha_3+3\alpha_4+2\alpha_5+\alpha_6$, then $2=l_1+2l_2+2l_3+3l_4+2l_5+l_6$, and the only new possibilities are $(1,0,0,0,0,1)$,   $(0,0,1,0,0,0)$ and $(0,0,0,0,1,0)$. But in the three cases, the dimension of the \emph{corner} $S_2$ is strictly greater than 1, since the root space related to the root $\alpha_1+\alpha_2+2\alpha_3+3\alpha_4+2\alpha_5+\alpha_6$ would be contained in $S_2$ too.
}\end{re}

We need     a more specific model of the $\Z$-grading on $\sig{14}$, in order to 
obtain refinements   whose complexified gradings are, respectively,   $\Gamma_{12}$ and $\Gamma_{10}$.

 \subsection{The model}  
Consider $L_0= \mathfrak{su}_{5,1}\oplus\R I_6$.  
We can describe the semisimple part\footnote{It is convenient to describe the semisimple part separately, because the center of the real Lie algebra 
$\{x\in \mathfrak{gl}(V): b(xu,v)+b(u,xv)=0 \;\forall u,v\in V\}$ is $\R\ii I_6$.}   as 
 \begin{equation*}
L'_0=[L_0,L_0] = \{x\in \mathfrak{sl}(V): b(xu,v)+b(u,xv)=0 \;\forall u,v\in V\},
\end{equation*} 
where $b$ is the hermitian form (linear in the first variable and conjugate-linear in the second one) given, relative   to a fixed $\C$-basis   $B_V=\{e_1,\dots,e_6\}$ of $V$, by
\begin{equation*}
b:   V\times V   \longrightarrow   \C,\qquad 
  b(\Sigma s_ie_i,\Sigma t_ie_i)  =    s_1 \bar{t}_1 + \dots + s_5 \bar{t}_5  -s_6 \bar{t}_6.
\end{equation*}
First, note that the hermitian form $b$ induces the conjugate-linear ($\phi(s v)=\overline{s}v$) $L'_0$-module isomomorphism ($\phi(x\cdot v)=x\cdot\phi(v)$)
\begin{equation*}
\begin{array}{cccl}
\phi\colon & V & \longrightarrow & V^*\\
& v & \longmapsto & b(-,v),
\end{array}
\end{equation*}
for $s\in\C$, $x\in L'_0$, $v\in V$.
This can be extended to another conjugate-linear $L'_0$-module isomorphism
\begin{equation*}
\begin{array}{cccl}
\tilde{\phi}: & \wedge^3V & \longrightarrow & \wedge^3V^*\\
& v_1\wedge v_2 \wedge v_3 & \longmapsto & \phi(v_1) \wedge \phi(v_2) \wedge \phi(v_3).
\end{array}
\end{equation*}
Second, our choice of basis $B_V$ allows us to identify  $S_0=\mathfrak{gl}(V)$ with $\mathfrak{gl}_6(\C)$ and $\wedge^6V$ with $ \C$ (by $e_1\wedge\dots\wedge e_6\mapsto1$). Then we have the $S_0$-module isomorphism
\begin{equation*}
\begin{array}{cccl}
\psi: & \wedge^3V & \longrightarrow & (\wedge^3V)^*\\
& v_1\wedge v_2 \wedge v_3 & \longmapsto & \psi_{v_1, v_2, v_3},
\end{array}
\end{equation*}
 where
 \begin{equation*}
\begin{array}{cccl}
\psi_{v_1, v_2, v_3}: & \wedge^3V & \longrightarrow & \C\equiv \wedge^6V\\
& v_4\wedge v_5 \wedge v_6 & \longmapsto & v_1 \wedge v_2 \wedge v_3 \wedge v_4\wedge v_5 \wedge v_6.
\end{array}
\end{equation*}
Third, recall that   $\wedge^3V^*$ and $(\wedge^3V)^*$ can be naturally identified by means of the $S_0$-module isomorphism
\begin{equation*}
\begin{array}{cccl}
\rho: & \wedge^3V^* & \longrightarrow & (\wedge^3V)^*\\
& f_1\wedge f_2 \wedge f_3 & \longmapsto & \rho_{f_1, f_2, f_3},
\end{array}
\end{equation*}
 where
 \begin{equation*}
\begin{array}{cccl}
\rho_{f_1, f_2, f_3}: & \wedge^3V & \longrightarrow & \C\\
& v_1\wedge v_2 \wedge v_3 & \longmapsto & \det(f_i(v_j)).
\end{array}
\end{equation*}
We are taking $v_i\in V$, $f_i\in V^*$.

We consider the  composition of the maps to obtain an $S_0$-module isomorphism 
\begin{equation*}
\rho^{-1} \psi : \wedge^3V \longrightarrow \wedge^3V^*,
\end{equation*}
which is in particular an $L_0$-module isomorphism (${L_0}^\C=S_0$). By composing the inverse of this map and the above (conjugate-linear) $L'_0$-module  isomorphism $\tilde{\phi}$, we obtain a conjugate-linear $L'_0$-module isomorphism
\begin{equation*}
\psi^{-1} \rho ~\tilde{\phi} : \wedge^3V \longrightarrow \wedge^3V.
\end{equation*}
It has order two, so $\wedge^3V$ is the sum of the eigenspaces corresponding to the eigenvalues $\pm1$.
We will construct a real form $L=\oplus L_n$ of $S_V$, with $L_1$ one of the eigenspaces (and, after identifying with the dual, $L_{-1}$ the other one).
Note that the (real) dimension of our $L'_0$-module $\wedge^3V$ is $ 2 \dim_{\C} (\wedge^3V) = 40$
and its complexification is the (non-irreducible) $S'_0 =[S_0,S_0]$-module $\wedge^3V\oplus \wedge^3V$.

\begin{lm}
Under the previous assumptions, the eigenspaces $ \mathrm{Ker} (\psi^{-1} \rho ~\tilde{\phi}\mp\id)$ are two $L'_0$-submodules of $\wedge^3V$, each one of dimension $ 20$.  
\end{lm}
\begin{proof}
 We   use the notation $e_{i_1\dots i_s}:=e_{i_1}\wedge \dots \wedge e_{i_s}$, so that  $\{e_{ijk},\ii e_{ijk}: 1\le i<j<k\le6\}$ is a real basis of $\wedge^3V$. Besides we denote   $s_i=1$ if $i=1,\dots 5$ and $s_6=-1$. A basis of the real space $\mathrm{Ker}  (\psi^{-1} \rho ~\tilde{\phi}-\id)$ is provided by
\begin{equation}\label{eq_basedeL1}
\left\{
\begin{array}{l}
 e_{i j k} -\mathrm{sgn}(\sigma)\,s_i s_j s_k\, e_ {lmn},\vspace{6pt}\\    
 \ii(e_{i j k}+\mathrm{sgn} (\sigma)\,s_i s_j s_k\, e_ {lmn}) 
 \end{array}\right|\left.
 \begin{array}{l}
 \sigma=\scriptsize{\begin{pmatrix}1&2&3&4&5&6\\i &j&k&l&m&n\end{pmatrix}}\in S_6\vspace{1pt}\\  i<j<k,\, l<m<n,\, i<l
  \end{array}
 \right\}.
\end{equation}
Similarly 
 $\{e_{i j k}+ \mathrm{sgn}(\sigma)\,s_i s_j s_k\, e_ {lmn}, \ii(e_{i j k}- \mathrm{sgn} (\sigma)\,s_i s_j s_k\, e_ {lmn})\}
$, with the permutation $\sigma$ as above, gives a basis of the eigenspace related to   $-1$. 
\end{proof}

(This lemma follows being true if we take the hermitian form $b$ of signature $(3,3)$, but it is false for signature $(4,2)$.)

We have already the tools to provide a convenient model of the real form.
\begin{pr} 
Take the real vector space 
$$
L=L_{-2}\oplus  L_{-1}\oplus L_{0}\oplus L_{1}\oplus L_{2}\subset S_V,
$$
for $L_0= \mathfrak{su}_{5,1}\oplus\R I_6$ and
\begin{equation}\label{eq_L_i}
\begin{array}{l}
L_1 := \{v\in \wedge^3V : \rho \tilde{\phi}(v)=\psi(v)\},\\
L_2  :=  [L_1,L_1],\\
L_{-1} := \tilde{\phi}\big(\{v\in \wedge^3V : \rho \tilde{\phi}(v)=-\psi(v)\}\big),\\
L_{-2}  :=  [L_{-1},L_{-1}].
\end{array}
\end{equation}
Then $L$ is a real form of $S_V\cong\e6$ of signature $-14$.
\end{pr}

\begin{proof}
Note that $L_i$ are  not only $L'_0$-modules but $L_0$-modules, since $I_6$ acts    as the identity on $V$, and hence,  as $n\id_{\wedge^nV}$  and $-n\id_{\wedge^nV^*}$ for $n=3,6$.  
If $\{e_i^*\}_{i=1}^6\subset V^*$ is the dual basis of $\{e_i\}_{i=1}^6$ (i.e., $e_i^*(e_j)=\delta_{ij}$), and 
we denote by
$e^*_{i_1\dots i_s}:=e^*_{i_1}\wedge \dots \wedge e^*_{i_s}$, it is easily checked that 
$$L_{-1}=\langle\{e^*_{i j k}- \mathrm{sgn}(\sigma)\,s_i s_j s_k\, e^*_{lmn}, \ii(e^*_{i j k}+ \mathrm{sgn} (\sigma)\,s_i s_j s_k\, e^*_{lmn})\}\rangle
,$$
with $\si\in S_6$ as above. Also,
 $L_{2}=\R\ii e_{123456}$ and $L_{-2}=\R\ii e^*_{123456}$. In particular, $L_i^\C=S_i$ and $L^\C=S=\e6$. To conclude that $L$ is indeed a real form of $S$, we need to check 
 that $[L,L]\subset L$. Note that  $[L_2,L_{-2}]=\R I_6$, and $[L_1,L_{-2}]\subset[[L_1,L_{-1}],L_{-1}]$ by the Jacobi identity, so that  the only non-trivial fact to be checked is that $[L_1,L_{-1}]\subset L$. This computation requires of the products in \cite[\S3.4]{Weyle6} recalled in Eq.~\eqref{eq_producots}. 
 According to them, the map $F=\ad(e^*_{ijk}-\mathrm{sgn} (\sigma)\,s_i s_j s_k\, e^*_{lmn})\in\ad(L_{-1})$, for an arbitrary permutation $\sigma\in S_6$, 
 satisfies 
 $$
 \begin{array}{l}
 F( e_{ijk}-\mathrm{sgn} (\sigma)\,s_i s_j s_k\, e_{lmn})=0,\\
F( {\ii(e_{ijk}+\mathrm{sgn} (\sigma)\,s_i s_j s_k\, e_{lmn})})\in\R \ii (E_{ii}+E_{jj}+E_{kk}-E_{ll}-E_{mm}-E_{nn} ) \subset L'_0,\\
F({e_{ijl}-\mathrm{sgn} (\sigma')\,s_i s_j s_l\, e_{kmn}})\in\langle E_{kl}-s_ls_k\,E_{lk},\ii(E_{kl}+s_ls_k\,E_{lk})\rangle\subset L,\\
F({\ii(e_{ijl}+\mathrm{sgn} (\sigma')\,s_i s_j s_l\, e_{kmn}}))\in\langle E_{kl}-s_ls_k\,E_{lk},\ii(E_{kl}+s_ls_k\,E_{lk})\rangle\subset L,
 \end{array}
 $$
 for $\si'=(34)\circ\si$, and hence $\ad(L_{-1})(L_1)\subset L$, as required.
  
 In order to compute the signature of the obtained real form $L$, note that $\kappa(I_6,I_6)=\tr(\ad^2I_6)=2(3^2\dim L_1+6^2\dim L_2)>0$. Besides the nilpotent pieces $L_n$ with $n\ne0$ do not contribute to the computation of the signature, so that the signature of $L$ coincides with
 $\sign\kappa\vert_{L_0}=1+\sign\kappa\vert_{L'_0}$ (the center of $L_0$ and the derived algebra $L_0'$ are orthogonal). But $\kappa\vert_{L'_0}$ is a positive multiple of the Killing form $\kappa_{L'_0}$ by Lemma~\ref{lema_kil}, so that  $\sign\kappa\vert_{L'_0}=\sign(\mathfrak{su}_{5,1})=-15$  and the signature of $L$ turns out to be $-14$.
 \end{proof}
  
\begin{re} \rm{
The existence of an $L_0$-module $L_1$ such that ${L_1}^\C\cong \wedge^3V$ is well-known. It is usually said that the complex
$\mathfrak{su}_{5,1}$-module $\wedge^3V$ determines a \emph{real} representation. This does not happen for the module $V$, i.e., there does not exist an $L_0$-module whose complexification is isomorphic to the $S_0$-module $V$. The hypotheses to be satisfied by a complex module to determine   a real representation can be consulted in \cite[\S8]{libroreales}.  A more specific construction of $L_1$, quite similar to the  one in Eq.~\eqref{eq_L_i}, has been  sketched in \cite[p.~425-426]{parabolic}, although without using such existence to construct  real forms of $\e6$.
}\end{re} 
 
 \subsection{The gradings}
 A suitable description of the $\Z^2\times\Z_2^3$-grading  $\Gamma_{12}$, adapted to the model of $\e6$ in Eq.~\eqref{eq_la5grad}, is developed in \cite[\S4.5]{Weyle6}. 
Take   $\theta\colon S\to S$  defined by:
$$
\begin{array}{l}
    \theta(e_{\sigma(1)}\wedge e_{\sigma(2)}\wedge e_{\sigma(3)})
   := \text{sg}(\sigma)\, \ii e_{\sigma(4)}\wedge e_{\sigma(5)}\wedge e_{\sigma(6)}, \\
   \theta(e^*_{\sigma(1)}\wedge e^*_{\sigma(2)}\wedge e^*_{\sigma(3)})
   := -\text{sg}(\sigma) \,\ii e^*_{\sigma(4)}\wedge e^*_{\sigma(5)}\wedge e^*_{\sigma(6)}, \\
    \theta(s I_6+x):=s I_6-x^t,   \\
  \theta|_{S_2\oplus S_{-2}} :=-\id,
\end{array}
$$
for any $s\in \C$, $x\in\mathfrak{sl}(V)$, $\sigma\in S_6$. It is checked in \cite[\S3.4]{Weyle6} that $\theta$ is an (outer) order 2 automorphism of $S$ fixing a subalgebra of type $\mathfrak c_4$.
 For each $A\in\text{GL}(V)$, let   $ {\varphi_A}\colon S\to S$ be the linear map defined by
$$
\begin{array}{ll}
{\varphi_A}(u_1\wedge\dotsc\wedge u_r) := Au_1\wedge\dotsc\wedge Au_r,    \\
{\varphi_A}(f_1 \wedge\dotsc\wedge f_r) :=(A\cdot f_1) \wedge\dotsc\wedge (A\cdot f_r),    \\
{\varphi_A}(s I_6+x) :=s I_6+A xA^{-1},
\end{array}
$$
for any $u_i\in V$, $f_i\in V^*$, $s\in \C$
and $x\in\mathfrak{sl}(V)$,
where $A\cdot f_i\in V^*$ is given by $(A\cdot f_i)(v)=f_i (A^{-1}v)$. Again $ {\varphi_A}$ is an (inner) automorphism of $S$. Take the invertible matrices 
$$
\begin{array}{ll}A_1=\diag({-1,-1,1,1,1,1}),&A_2=\diag({-1,1,-1,1,1,1}),\\
A_3=\diag({-1,1,1,-1,1,1}),&A_4=\diag({-1,1,1,1,-1,1}),\end{array}
$$
and the order 2   automorphisms $F_i=\varphi_{A_i}\in\aut(S)$.\smallskip

Then  $\Gamma_{12}$ is produced when refining the $\Z$-gra\-ding~\eqref{eq_la5grad} on $S$ by considering the simultaneous eigenspaces relative to all the automorphisms
 in $Q=\{F_1,\dots,F_4,\theta\}\subset\aut(S)$.  
 In order to prove that $\Gamma_{12}$ is inherited by $\sig{14}$, it is sufficient to prove that  $\theta(L)\subset L$ and  $F_i(L)\subset L$ for all $i=1,\dots,4$, so that the restriction to $L$ of the elements in $Q$  are automorphisms of $L$. 
These are straightforward computations.  
For instance, $F_1$ acts on $e_{i j k} -\mathrm{sgn}(\sigma)\,s_i s_j s_k\, e_ {lmn}\in L_1$ with eigenvalue $-1$ if the set $\{i,j,k\}\cap \{1,2\}$ has cardinal equal to $1$ and with eigenvalue 1 otherwise; and so on.
\medskip

Now we deal with the $\Z^2\times\Z_2^3$-grading $\Gamma_{10}$. In this case, it is  convenient for us to use neither the description of the grading in \cite[\S4.6]{Weyle6}, nor in \cite[\S5.3]{e6}, but we need to find an equivalent  description of $\Gamma_{10}$  compatible with some real form of signature $-14$.
Take the element
\begin{equation*}
E=\tiny{\left(\begin{array}{cccccc}
0 & 0 & 0 & 0 & 0 & 0\\
0 & 0 & 0 & 0 & 0 & 0\\
0 & 0 & 0 & 0 & 0 & 0\\
0 & 0 & 0 & 0 & 0 & 0\\
0 & 0 & 0 & 0 & 0 & \ii\\
0 & 0 & 0 & 0 & -\ii & 0
\end{array}\right)}\normalsize=\ii(E_{56}-E_{65})\in\mathfrak{su}_{5,1}. 
\end{equation*}  
Note that the endomorphism $\ad E\colon L\to L$ is diagonalizable with eigenvalues $\pm2,\pm1,0$, producing a $5$-grading on $L$. Moreover,
$\{x\in L: [E,x]=2x\}=\R(\ii(E_{55}-E_{66})+E_{56}+E_{65})$ has dimension 1, so that it is a contact grading. 
By Remark~\ref{re_solouncontacto}, the $\Z$-grading on $S$ produced by the 
eigenspace decomposition of $\ad E\colon S\to S$ is equivalent (in fact,  \emph{isomorphic}) to that one in Eq.~\eqref{eq_la5grad}. 
Moreover, both $\Z$-gradings on $L$ are compatible:  as $E\in L_0$, then $\ad E(L_n)\subset L_n$, so that we have a $\Z^2$-grading on $L$ given by $L_{(m,n)}=\{x\in L_n: [E,x]=mx\}$. Now, each of the automorphisms   $q\in\{\theta,F_1,F_2\}$ satisfies $q(E)=E$ and $q(I_6)=I_6$, so that   $q(L_{(m,n)})\subset L_{(m,n)}$. This means that the group generated by $ \{\theta,F_1,F_2\}$ induces a $\Z_2^3$-grading on $L$ compatible with the above grading over the group $\Z^2$, getting a $\Z_2^3\times \Z^2$-grading on $L$ whose complexified 
grading is just $\Gamma_{10}$. In particular, it is fine.

 \section{A $\Z_2^7$-grading}\label{todoZ2}
 
It was proved in  \cite[Proposition~11]{sig26} the existence of this fine $\Z_2^7$-grading on $\sig{14}$. We will  briefly describe the  main ideas. Our interest is also to study the basis provided by the grading.

 Let $ {B}_{\e6}= \{h_i,e_\alpha,f_\alpha: i=1,\dots,6, \alpha\in\Phi^+\} $ be a Chevalley basis of $S=\e6$, so that 
 the       structure constants are integers, $\Phi$ is the root system relative to a Cartan subalgebra 
 $\mathfrak{h}=\langle h_i: i=1,\dots,6\rangle$, 
 and $e_\alpha\in L_{\alpha}$, $f_\alpha\in L_{-\alpha}$ 
 and $h_i\equiv h_{\a_i}=[e_{\alpha_i},f_{\alpha_i}]$, for $\{\alpha_i : i=1,\dots,6\}$ a set of simple roots of $\Phi$.
 The $\Z_2^7$-grading $\Gamma_{13}$ on $S$ is given by the
 simultaneous diagonalization relative to the following MAD-group of the automorphism group:  
 \begin{equation*} 
 \{t,\omega t: t\in T, t^2=1 \},
 \end{equation*}
where the maximal torus $T$ is given by Eq.~\eqref{eq_toro} and $\omega\in\aut(S)$ 
is   the involutive automorphism  determined by 
$$
\omega(e_{\alpha_i})=-f_{\alpha_i},\quad
\omega(f_{\alpha_i})=-e_{\alpha_i},\quad
\omega(h_{ i})=-h_{ i},
$$
for all $i=1,\dots,6$.
Let $\sigma_0\colon S\to S$ be a conjugation such that $\sigma_0\vert_{{B}_{\e6}}=\id$. 
 The algebra $S^{\sigma_0}$
 ($\R$-spanned by  ${B}_{\e6}$) is a split real form of $S$ which inherits the $\Z_2^7$-grading, since $\si_0$ commutes with $\omega$ and with every order two automorphism in the torus. By Proposition~\ref{pr_nuestrometodo}, 
  the algebra $S^{\sigma_0\omega t}$ inherits the $\Z_2^7$-grading too  for any $t= t_{s_1, \dots, s_6}$ with $s_i^2=1 $. The signature of the real form $S^{\sigma_0\omega t}$ obviously depends on $t$. To be precise,
  $\sign(S^{\sigma_0\omega t})=78-2\dim\fix(t)$ by Eq.~\eqref{eq_signapartirparte fija}, since ${\sigma_0\omega}$ is a compact conjugation commuting with $t$.
But $\dim\fix(t)-6$ is equal to the cardinal  of  
$
\{(k_1,\dots,k_6):\sum_{i=1}^6k_i\a_i\in\Phi,s_1^{k_1}\dots s_6^{k_6}=1\}.
$ 
For instance, if we choose $t=t_{-1,1,1,1,1,1}$,  then $\dim\fix(t)=46$ because there are  20/16 positive roots with $k_1$  even/odd (respectively), that implies
that $S^{\sigma_0\omega t}$ is a real form of signature -14
with a fine $\Z_2^7$-grading.\smallskip

 Let us take a closer look at the properties of the obtained grading, and, more concretely, at the basis provided by it.
 \begin{co}\label{co_basemona}
 There is a basis $\mathcal B=\{u_i:i=1,\dots,78\}$ of $\sig{14}$ satisfying the following properties:
 \begin{itemize}
 \item[a)] $\mathcal B$ is an orthogonal basis for the Killing form.
 \item[b)] Every element in $\mathcal B$ is semisimple.
 \item[c)]  
 If $[u_i,u_j]=\sum_k f^{ijk}u_k$, then the structure constants  $f^{ijk}$   are rational numbers and completely antisymmetric in the three indices.
 
 \end{itemize}
 \end{co}
 
 As mentioned in Introduction, this generalizes what happens with the Pauli matrices as well as with the Gell-Mann matrices. Note that Gell-Mann matrices provide a $\Z_2^3$-grading on $\mathfrak{su}_3$ with zero neutral component, one homogeneous component spanned by two commuting Gell-Mann matrices and each of the remaining 6 homogeneous components  spanned by one of the remaining matrices.
 
 \begin{proof}
 Recall that $\kappa\vert_{\mathfrak h_0}$ is negative definite for  
 $\mathfrak h_0=\sum_{j=1}^6\R \ii h_j  $ (a real form of $\mathfrak h$),
 which is the  only homogeneous component of the $\Z_2^7$-grading on $\sig{14}$ of dimension strictly greater than 1. 
 Every element in $\mathfrak h_0$ is semisimple with purely imaginary spectrum. 
 Take $\mathcal B_0=\{\ii h_j':j=1,\dots,6\}$ the following orthogonal basis of $\mathfrak h_0$:
 $$\begin{array}{l}
 h'_1=h_1,\,h'_2=h_2, \,h'_3=h_1+2h_3, \,h'_4=2h_1+3h_2+4h_3+6h_4, \\
 h'_5= 2h_1+3h_2+4h_3+6h_4+5h_5,\,h'_6= 2h_1+3h_2+4h_3+6h_4+5h_5+4h_6.
 \end{array}$$
 The required  basis is provided by this one joint with   homogeneous elements of the grading, namely,
 $$
\mathcal B= \mathcal B_0 \cup\{e_{\a}-f_{\a},\ii(e_{\a}+f_{\a}):\a\in\Phi_0^+\}\cup\{e_{\a}+f_{\a},\ii(e_{\a}-f_{\a}):\a\in\Phi_1^+\},
 $$
 where $\Phi_i^+:=\{\a=\sum k_i\a_i\in\Phi^+:k_1=i\}$ if $i=0,1$. 
 Clearly $\mathcal B$ is orthogonal, recalling that $\kappa(e_\a,e_\b)=0=\kappa(f_\a,f_\b)$ for all $\a,\b\in\Phi^+=\Phi^+_0\cup\Phi^+_1$, and that $\kappa(e_\a,f_\b)\ne0$ only for $\b=\a$. Moreover, the first 46 elements in $\mathcal B$ have negative norm, and the last 32 elements in $\mathcal B$ have positive norm, thus recovering the fact of being the signature of $\kappa$ equal to $-14$.
 
Item b) is a direct consequence of \cite[Lemma~1]{e6}:  If  a fine  grading on a complex simple Lie algebra satisfies that the universal grading  group is finite, then every homogeneous element is semisimple.  
 
 In order to study the structure constants, recall that   $[e_\a,e_\b]=N_{\a,\b}e_{\a+\b}$ for $|N_{\a,\b}|=1+p_{\a,\b}$, where   $p_{\a,\b}$ is the greatest integer such that $\b-p_{\a,\b}\a$ is a root \cite[\S25.4]{Draper:Humphreysalg}. Here the used notation is $e_{-\a}=f_\a$ for $\a\in\Phi^+$.
 In the case of $\e6$, the strings have length at most 2: if $\a,\b,\a+\b\in\Phi$, then $\a+2\b\notin\Phi$ and $\a-\b\notin\Phi$. Hence $[e_\a,e_\b]=\pm e_{\a+\b}$ if $\a+\b$ is a root, and $0$ otherwise. Now, most of the structure constants $f^{ijk}$ are integers ($\a(h'_j)\in\Z$), except for $[e_\a+f_\a,\ii(e_\a-f_\a)]=-2\ii \sum_jk_jh_j\in\Z[\frac1{60}]\mathcal B_0$ if $\a=\sum_jk_j\a_j\in\Phi_1^+$ (and similarly for $\Phi_0^+$).\footnote{ If we change $\mathcal B$ by an orthonormal basis,   then   $f^{ijk}\notin\mathbb Q$, but $f^{ijk}\in\mathbb Q[\sqrt{2},\sqrt{3},\sqrt{5}]\subset\R$. }
 
 Finally, the equality $N_{\a,\b}=N_{\b,\gamma}=N_{\gamma,\a}$ for roots such that $\a+\b+\gamma=0$ (all of them have necessarily the same length) proves the antisymmetry of the structure constants.   
 \end{proof}
 
 \begin{re}{\rm
 A basis with the same properties as in Corollary~\ref{co_basemona} can be found in $\mathfrak{e}_{6,-78}$, in $\mathfrak{e}_{6,6}$ and in $\mathfrak{e}_{6,2}$ too.
 }\end{re}

 \section{No more fine gradings}
 
 As a consequence of the above sections, the gradings $\Gamma_{3}$, $\Gamma_{7}$, $\Gamma_{8}$, $\Gamma_{10}$, $\Gamma_{12}$ and $\Gamma_{13}$ are inherited by $\sig{14}$. The purpose now is to prove that these are the only gradings on $\sig{14}$ whose 
 complexified  grading  is fine.
   
  The gradings $\Gamma_{1}$, $\Gamma_{2}$, $\Gamma_{5}$, $\Gamma_{6}$ and $\Gamma_{9}$ cannot be inherited by $\sig{14}$
by Proposition~\ref {pr_lemano}, taking into account the data
    $  \dim_\mathbb{C}S_e\pm\dim_\mathbb{C}\sum_{\stackrel{2g=e}{g\ne e}}S_g$    in  Table~\ref{tablafinasdee6}.
    
  We will prove that  $\sig{14}$    inherits neither $\Gamma_{4}$, nor $\Gamma_{11}$, nor $\Gamma_{14}$, by using ad-hoc arguments for each case.

 \subsection{No fine $\Z_4^3$-grading}
Proposition~\ref {pr_lemano} implies   that $\sig{26}$ and $\sig{78}$ do not inherit $\Gamma_{14}$, but it does not say anything about $\sig{14}$.
We will prove that $\sig{14}$ neither inherits $\Gamma_{14}$   by reductio ad absurdum. We will use the knowledge of some features about $\Gamma_{14}$ extracted from \cite[\S5.4]{e6}, joint with some   representation theory of real Lie algebras.

 Recall first that the complex exceptional Lie algebra $\e6=S$ has a fine grading entitled $\Gamma_{14}:S=\oplus_{g\in\Z_4^3}S_g$ such that
 its coarsening $S=S_{\bar 0}\oplus S_{\bar 1}\oplus S_{\bar 2}\oplus S_{\bar 3}$ defined by 
 $S_{\bar i}=\oplus_{h\in\Z_4^2}S_{(\bar i,h)}$ satisfies
 $$
S_{\bar 0}=\frak{a}_3\oplus\mathfrak{sl}(V),
\quad S_{\bar 1}=V(2\lambda_1)\otimes V
,\quad S_{\bar 2}=V(2\lambda_2)\otimes\C,
\quad S_{\bar 3}=V(2\lambda_3)\otimes V,
$$
 where $V$ is now a two-dimensional (complex) vector space, 
 and $\lambda_i$'s denote again the fundamental weights but this time 
 of the simple Lie algebra $\frak{a}_3=\mathfrak{sl}_4(\C)$ (i.e., the maximal weights of its basic representations, defined by $\lambda_i(h_{\alpha_j})=\delta_{ij}$ if $i,j=1,2,3$). Thus, 
 $\dim
 S_{\bar 0}=15+3=18$ while $\dim  S_{\bar i}=20$ for $\bar i=\bar 1,\bar 2,\bar 3$.
 The $\Z_4^2$-grading on $\mathfrak{sl}_4(\C)\subset S_{\bar 0}$ is given by the   generalized Pauli matrices according to  the following degree assignment: 
 \begin{equation}\label{eq_Z42}
 \deg \left(\begin{pmatrix}
 0&0&0 & 1  \cr
 1&0&0&0 \cr
 0&1&0&0 \cr
 0&0&1&0
 \end{pmatrix}^i\begin{pmatrix}
 1&0&0 & 0  \cr
 0&\imag&0&0 \cr
 0&0&-1&0 \cr
 0&0&0&-\imag
 \end{pmatrix}^j\right)=(\bar i,\bar j),
\end{equation}
 for all $i,j\in\{0,1,2,3\}$.
 
 Suppose that there exists a grading on $L=\oplus_{g\in\Z_4^3}L_g$, a real form of $S$,  
 such that the complexified grading is $\Gamma_{14}$. That is, $L_g\otimes_\R\C=S_g$. 
  In particular, $L$ has a $\Z_4$-grading $L=L_{\bar 0}\oplus L_{\bar 1}\oplus L_{\bar 2}\oplus L_{\bar 3}$ defined by $L_{\bar i}=\oplus_{h\in\Z_4^2}L_{(\bar i,h)}$, and $L_{\bar i}\otimes_\R\C=S_{\bar i}$. 
 This means that $L_{\bar 0}$ is a real form of $\mathfrak{sl}_4(\C)\oplus\mathfrak{sl}_2(\C)$. Hence $L_{\bar 0}= L_{\bar 0}^1\oplus L_{\bar 0}^2$, for 
 $  L_{\bar 0}^1$ and $L_{\bar 0}^2$ real forms of $\mathfrak{sl}_4(\C)$ and $\mathfrak{sl}_2(\C)$ respectively. (In general, if $\si$ is the conjugation of a complex algebra $S=L^\C$ with $S^\si=L$, and $S$ is sum of two   ideals $S=S^1\oplus S^2$, then $L=L^1\oplus L^2$ for $L^i=(S^i)^\si$.)
 
 Among the  $5$  real forms of $\mathfrak{sl}_4(\C)$, that is,
 $\mathfrak{su}_{4}$, $\mathfrak{su}_{3,1}$, $\mathfrak{su}_{2,2}$, $\mathfrak{sl}_4(\R)$ and $\mathfrak{sl}_2(\H)$  (see Section~\ref{subse_realforms}),
 the only ones which inherit the $\Z_4^2$-grading given by Eq.~\eqref{eq_Z42},
 are $\mathfrak{su}_{3,1}$ and $\mathfrak{su}_{2,2}$, according to \cite{Svobodova}.
 Thus,  
 $$
 L_{\bar 0}^1\in\{\mathfrak{su}_{3,1},\mathfrak{su}_{2,2}  \} ,\quad
 L_{\bar 0}^2\in\{\mathfrak{su}_{2},\mathfrak{sl}_2(\R)  \}.
 $$
 Note that $L_{\bar 1}$ is an $L_{\bar 0}^1$-module whose complexification $S_{\bar 1}$ is isomorphic to $2V(2\lambda_1)$ as $(L_{\bar 0}^1)\otimes_\R\C\cong \mathfrak{sl}_4(\C)$-module.

 We now recall   some basic facts on representations of real Lie algebras, for instance from \cite[2.3.14]{parabolic}.
 If $W$ is a complex vector space, we denote by $\overline{W}$ the new complex vector space with the same underground set and scalar multiplication given by $\C\times \overline{W}\to \overline{W}$,
 $(a+\mathbf{i} b,w)\mapsto (a-\mathbf{i} b)w$. For $\mathfrak{g}$ a real Lie algebra, $W$ is called a \emph{complex representation} of $\mathfrak{g}$ if there is  a homomorphism of real Lie algebras $\mathfrak{g}\to\mathfrak{gl}_\C(W)$.
 In this case, $\overline W$ is naturally a complex representation of $\mathfrak{g}$ called \emph{the conjugate representation}. As real representations, $W$ and $\overline W$ are always isomorphic (the identity map is an isomorphism), but this is not the case as complex representations in general.

 If $U$ is a (real) $\mathfrak g$-module, then $W=U^\C$ is a complex representation and in this case  the   map 
 $R\colon W\to\overline W$, $R(u_1+\ii u_2)=u_1-\ii u_2 $  ($u_i\in U$), provides an isomorphism of complex representations of $\mathfrak g$. This can be applied to our setting, since $S_{\bar 1}$ is an $L_{\bar 0}^1$-complex representation which is the complexification of $L_{\bar 1}$, so that $S_{\bar 1}\cong\overline{S_{\bar 1}}$ is a self-conjugate $L_{\bar 0}^1$-module. 
 On one hand we know that $S_{\bar 1}$ is isomorphic to $ 2V(2\lambda_1)$ as $\mathfrak{sl}_4(\C)$-module, and hence as $L_{\bar 0}^1$-module. 
 On the other hand, $ L_{\bar 0}^1\in\{\mathfrak{su}_{3,1},\mathfrak{su}_{2,2}  \}$. But for both algebras it is well known
 (\cite[p.~230]{parabolic} or \cite[Table~5]{libroreales})  that the conjugate representation   $\overline{V(2\lambda_1)}\cong V(2\lambda_3)$, which gives a contradiction since $2V(2\lambda_1)$ cannot be a self-conjugate complex $L_{\bar 0}^1$-representation.
 Hence we have proved that 
 
 \begin{pr}
 None of the  real forms of $\mathfrak e_6$ inherit $\Gamma_{14}$.
 \end{pr}
    \subsection{No inner fine $\Z^2\times \Z_2^3$-grading}

  The grading $\Gamma_{4}$ can be described as follows. The algebra $S=\e6$ is modeled by the Tits' construction
  $$
  S=\T(\OO^\C,M)=\der(\OO^\C)\oplus \,(\OO^\C_0\otimes M_0)\,\oplus\der( M),
  $$
  for the Jordan algebra $M={\M_s}^\C= \Mat_{3\times 3}(\C)^+$, where again the products are given by Eq.~\eqref{eq_TitsProduct}.
  If we consider the $\Z_2^3$-grading on $\OO^\C$ (complexified of the grading in \eqref{eq_gradenO}) and the $\Z^2$-grading on $M $ given by the assignment $\deg(E_{12})=(1,0)$ and $\deg(E_{23})=(0,1)$, the obtained $\Z^2\times \Z_2^3$-grading on $S$ is precisely  $\Gamma_4$ in Table~\ref{tablafinasdee6}. Observe that the $\Z^2$-grading induced on the Lie subalgebra $\der( M)\cong\mathfrak{sl}_3(\C)$ is precisely the root decomposition of $\mathfrak a_{2}$. 
  
 Suppose that there exists a grading on $\sig{14}=L=\oplus_{g\in \Z^2\times \Z_2^3}L_g$ such that the complexified grading is $\Gamma_{4}$. 
  By arguments as in Proposition~\ref{pr_lemano},  the signature of $L$ equals
   the signature of the restriction of the Killing form of $L$ to its subalgebra $\tilde L :=\oplus_{h\in   \Z_2^3}L_{(0,0,h)}$. Then we study $\tilde L $. Its complexification   is 
  $\tilde S=\oplus_{h\in   \Z_2^3}S_{(0,0,h)}$, the subalgebra fixed by the two-dimensional torus
   of the automorphism group of $S$ producing the $\Z^2$-grading, that is,
  $\tilde S=\der(\OO^\C)\oplus (\OO^\C_0\otimes H)\oplus  H$, 
  for $H=\langle E_{11}-E_{22},E_{22}-E_{33}\rangle $ seen as a subset of $ M_0$, but also of $\mathfrak{sl}_3(\C) = \der (M)  $. This second piece $H\subset \der (M)$  is a two-dimensional centre of $\tilde S$, while the complementary  subspace  $[\tilde S,\tilde S]=\der(\OO^\C)\oplus (\OO^\C_0\otimes H)$  is a simple subalgebra of $\tilde S$ isomorphic to $\d4=\mathfrak{so}_8(\C)$.

  Thus $[\tilde L ,\tilde L ]$ is a simple subalgebra of $\tilde L $ and a real form of $\d4$. By Section~\ref{subse_realforms}, the only possibilities for $[\tilde L ,\tilde L ]$ are $\mathfrak{so}_{8}$, $\mathfrak{so}_{7,1}$, $\mathfrak{so}_{6,2}$, $\mathfrak{so}_{5,3}$, $\mathfrak{so}_{4,4}$, whose signatures are, respectively, $-28,-14,-4,2,4$.

In the proof of Proposition~\ref{pr_lemano}, it is not only proved that   
  $$
\sign  \kappa_L=\sign  \kappa\vert_{\tilde L }= \sign  \kappa\vert_{L_e} +\sign  \kappa\vert_{\sum_{2g=e,g\ne e}L_g},
$$
  but also that $ \kappa\vert_{L_e}$ is positive definite, 
  so that $\sign  \kappa\vert_{L_e}=\dim L_e$.
  In our case $  L_e=\tilde L_e=H\oplus [\tilde L ,\tilde L ]_e$, so that  $\sign  \kappa\vert_{L_e}=2+\sign  \kappa\vert_{[\tilde L ,\tilde L ]_e}$ and hence
  $$
\sign  \kappa_L=2+\sign  \kappa\vert_{[\tilde L ,\tilde L ]}.
$$
But  $\sign  \kappa\vert_{[\tilde L ,\tilde L ]}= \sign ({[\tilde L ,\tilde L ]})$ by Lemma~\ref{lema_kil}, which gives a contradiction as $-14\notin 2+\{-28,-14,-4,2,4\}$. (By the way, the number $2$ neither belongs to such set.)
 Hence, 
 we conclude that

  \begin{pr}
Neither $\sig{14}$ nor $\mathfrak e_{6,2}$   inherit $\Gamma_{4}$.
 \end{pr}

   \subsection{No $\Z_4\times \Z_2^4$-grading}
   
   There are two classes of order 2 outer automorphisms of $S=\e6$: those ones fixing a subalgebra of type $\c4=\mathfrak{sp}_8(\C)$ and those ones fixing a subalgebra of type $\f4$ (see Table~\ref{tabladeautomorfysignaturas}). For the grading $\Gamma_{11}$, all the order 2 outer automorphisms belonging to the MAD-group of automorphisms producing the grading are of the first type \cite[\S5.3]{e6}. Thus, we need to know which of the real forms of $\c4$ are even parts of a $\Z_2$-grading on $\sig{14}$.  
   Take in mind that there are 4 real forms of $\c4=\mathfrak{sp}_{8}(\mathbb C)$, namely, $\mathfrak{sp}_{4}(\mathbb H)$, $\mathfrak{sp}_{3,1}(\mathbb H)$, $\mathfrak{sp}_{2,2}(\mathbb H)$ and $\mathfrak{sp}_{8}(\mathbb R)$, of signatures $-36$, $-12$, $-4$ and $4$, respectively  (Section~\ref{subse_realforms}).

   \begin{lm}\label{pr_hereda11}
If $L=L_{\bar0}\oplus L_{\bar1}$ is a $\Z_2$-grading on   $\sig{14}$   such that $L_{\bar0}$ is a real form of $\c4$, then $L_{\bar0}$ is a Lie algebra isomorphic to $\mathfrak{c}_{4,-4}=\mathfrak{sp}_{2,2}(\mathbb H)$.
\end{lm}

\begin{proof} 
 By Lemma~\ref{lema_parejas}, $\sign(L^{-1})=2\sign(L_{\bar0})+14$. As $L^{-1}$ is also a real form of $\e6$,  its signature must belong to $\{6,2,-14,-26,-78\}$. But, taking into account that $\sign(L_{\bar0})\in\{4,-4,-12,-36\}$, the only true possibility is $\sign(L_{\bar0})=-4$. (The argument is a trivial computation: $2\cdot 4+14=22$ is not admissible and so on.) 
\end{proof}

In fact, it is not difficult to prove that there exists    a $\Z_2$-grading $L=L_{\bar0}\oplus L_{\bar1}$ on   $L=\sig{14}$ with $L_{\bar0}\cong \mathfrak{sp}_{2,2}(\mathbb H)$ (so that  $L^{-1}$ split), but it is not necessary for our purposes.

\begin{lm}\label{le_solounas}
Take the fine grading $\Gamma_{11}: S=\oplus_{g\in\Z_2^4\times \Z_4} S_g$ and consider a coarsening $S=S_{\bar0}\oplus S_{\bar1}$ with $S_{\bar0}\cong\mathfrak c_4$.
Let us denote by $\Gamma_{11}':S_{\bar0}=\oplus_{h\in\Z_2^3\times \Z_4} S_{(\bar0,h)}$  the $\Z_2^3\times \Z_4$-fine grading on $\c4$ obtained by restriction of $\Gamma_{11}$ to the even part $S_{\bar0} $ (perhaps after reordering the indices). 
The only real forms of $\c4$ which inherit $\Gamma_{11}'$ are those ones with signatures $4$ and $-12$.
\end{lm}

The   existence of real forms $\mathfrak{c}_{4,4}$ and $\mathfrak{c}_{4,-12}$ inheriting $\Gamma_{11}'$ was proved in \cite[Proposition~8]{sig26}. But precisely we are interested in that they are the only cases. The advantage of working with $\mathfrak{c}_4$ instead of $\e6$  is that all the computations can be done with matrices.

\begin{proof}
Recall from \cite[V.C]{sig26} that the grading $\Gamma_{11}'$
on the complex Lie algebra 
$ S_{\bar0}=\mathfrak{sp}_\C(8,C)=\{x\in\Mat_{8\times8}(\C): xC+Cx^t=0\}$, for
$C=\tiny\left(\begin{array}{cc}
0 & I_4 \\
-I_4 & 0 
\end{array}\right)$,
is given by the simultaneous diagonalization of $S_{\bar0}$ relative to the automorphisms
$\Ad A_i\in\aut(\mathfrak{sp}_\C(8,C))\cong \mathrm{PSp}_\C(8,C)$, $x\mapsto A_i x {A_i}^{-1}$,
for the invertible matrices
$$
\begin{array}{l}
A_1=\ii\left(\begin{array}{cccc}
0 &0 &  I_2 & 0\\
0 & 0 & 0 & \s_1 \\
 I_2 & 0 & 0 & 0\\
0 & \s_1 & 0 & 0
\end{array}\right) ,\\
A_2=\ii\diag\{{I_4,-I_4}\} ,\\
A_3=\diag\{{\sigma_1,\sigma_1,\sigma_1,\sigma_1}\},\\
A_4=\diag\{{1,-1,-\ii,\ii,1,-1,\ii,-\ii}\} ,
\end{array}
$$
 where we denote $\sigma_1= \tiny\left(\begin{array}{cc}
0&1\\
1&0
\end{array}\right)$.

Recall, also  from \cite[V.C]{sig26}, that there is a basis $B_0$  of $S_{\bar0}$ 
formed by homogeneous elements (simultaneous eigenvectors), all of them matrices with entries in the set $\{1,0,-1\}$. In particular $B_0\subset  \mathfrak{sp}_\R(8,C)=\{x\in\Mat_{8\times8}(\R): xC+Cx^t=0\}$, 
which is a split real form of $\c4$ that, obviously, inherits $\Gamma_{11}'$. 
(Take into account that a real algebra $L$ inherits a grading $\Gamma$ on $L^\C$ if and only if there is a basis of $L$ formed by homogeneous elements of $\Gamma$.)

If $\sigma_0$ denotes the conjugation fixing $\mathfrak{sp}_\R(8,C)$, then, by  Proposition~\ref{pr_nuestrometodo},   the set of real forms inheriting $\Gamma_{11}'$ is exactly
\begin{equation}\label{eq_posibles}
\{S^{\sigma_0\Ad A}:A\in\langle A_1,A_2,A_3,A_4\rangle, (\sigma_0\Ad A)^2=\id\}.
\end{equation}
Then the task is   to study the signatures of all the real forms in the set in Eq.~\eqref{eq_posibles}.
According to Eq.~\eqref{eq_signapartirparte fija},
$\sign( {S^\si})=36-2\dim\fix(\theta_\si)$, for $\theta_\si\in\aut(S_{\bar0})$ of order 2, commuting with $\si$, and  such that $\si\theta_\si$ is compact.
For instance, $\tau=\sigma_0\Ad(C)$ is a compact conjugation (\cite[Remark~2]{sig26}) which commutes with all the   conjugations in \eqref{eq_posibles}. Now we compute (tediously, but straightforwardly)
$$
\dim\mathrm{fix}( \Ad (CA))=\left\{\begin{array}{ll}
24&\textrm{if $A=A_1A_2A_3^sA_4^r$, for $s=0,1$, $r=0,1,2,3$},\\
16&\textrm{otherwise}.
\end{array}
\right.
$$
Hence,  the   signatures of the real forms in \eqref{eq_posibles} are:
$$
36-2\dim\mathrm{fix}( \Ad (CA))\in\{-12,4\}.
$$ 
\end{proof}

The immediate consequence is
   
  \begin{co}
  The real form $\sig{14}$ does not inherit $\Gamma_{11}$.
  \end{co}
  
  \begin{proof}
  By Lemma~\ref{pr_hereda11},  if  $\sig{14}$   inherited $\Gamma_{11}$, then  $\mathfrak{sp}_{2,2}(\mathbb H)$ would inherit $\Gamma_{11}'$. This would contradict Lemma~\ref{le_solounas}.
  \end{proof} 
    
\subsection{Conclusions about   gradings on $\sig{14}$.}

By summarizing the previous sections, we have proved our main result:

\begin{te}\label{th_main}
There are exactly 6 fine gradings on $\e6$ producing fine gradings on $\sig{14}$, namely, the
$\Z_2^3\times\Z_3^2$-grading $\Gamma_{3}$,  the $\Z_2^6$-grading $\Gamma_{7}$,  the $\Z\times\Z_2^4$-grading $\Gamma_{8}$,  the $\Z_2^7$-grading $\Gamma_{13}$,  the
$\Z\times\Z_2^5$-grading $\Gamma_{12}$  and the (outer) $\Z^2\times\Z_2^3$-grading $\Gamma_{10}$.

\end{te}

 As happened in the study of $\sig{26}$, the following situations have still to be studied if we want to have a complete knowledge of the fine gradings on $\sig{14}$ up to equivalence:
\begin{itemize}
\item[a)]  There could exist a fine grading  on $\sig{14}$ whose complexification would not be a fine grading on $\e6$.

\item[b)] There could be two fine gradings on $\sig{14}$ not isomorphic but with isomorphic complexifications.
\end{itemize}
These questions have mainly  mathematical interest. For physical purposes, an interesting topic could be to study the different bases provided by the gradings and their properties: are they involved in  some physical phenomenon? Taking into account that every fine grading on a simple Lie algebra over a finite group provides a basis of semisimple elements, and that we have described a fine grading over $\Z_2^3\times\Z_3^2$, what type of properties or processes  can be described or modeled by an abelian  group with 3-torsion? 

 \medskip 
 
 \begin{center}\textbf{Acknoledgements}\end{center}
 
 The authors thank Professors A.~Elduque for his read of the manuscript and M.~Atencia for his English language revision.



\begin{thebibliography}{99}


\bibitem[ADrGu14]{Weyle6}
D. Aranda, C.~Draper and V.~Guido.
{\sl Weyl groups of the fine gradings on $\e6$}.
 J.Algebra  \textbf{417}(1) (2014), 353--390.

\bibitem[At15]{Atanasov}
A.~Atanasov.
{\sl Graded Lie Algebras, Supersymmetry, and Applications}.
Paper consulted in the web page http://abatanasov.com/

\bibitem[BKR18a]{reales2}
 Y.~Bahturin, M.~Kochetov, and A.~Rodrigo-Escudero.
{\sl Classification of involutions on
graded-division simple real algebras,}
 Linear Algebra Appl. \textbf{546} (2018), 1--36.

\bibitem[BKR18b]{reales3}
Y.~Bahturin, M.~Kochetov, and A.~Rodrigo-Escudero.
{\sl Gradings on classical central simple real Lie algebras,}
   J.~Algebra  \textbf{506} (2018), 1--42.

\bibitem[BFGM06]{agujerosnegros}
  S.~Bellucci, S.~Ferrara, M.~Gunaydin and A.~Marrani. 
{\sl  Charge orbits of symmetric special geometries and attractors}. 
 Internat. J. Modern Phys. A 21 (2006), no. 25, 5043--5097. 


\bibitem[CalDrM10]{reales}
A.J.~Calder\'{o}n, C.~Draper and C.~Mart\'{\i}n.
{\sl  Gradings on the real forms of $\frak{g}_2$ and $\frak{f}_4$}.
J.~Math. Phys.   \textbf{51}(5) (2010), 053516, 21 pp.

\bibitem[CS]{parabolic}
A.~\v{C}ap and J.~Slov\'{a}k.
Parabolic Geometries I, Background and General Theory.
Mathematical Surveys and Monographs, Vol 154. Amer. Math. Soc. 2009.

\bibitem[Ca27]{CartanAutomorfismos}
\'{E}.~Cartan.
{\sl La g\'eom\'etrie des groupes simples}.
Ann. di Mat. \textbf{4} (1927),   209--256.

\bibitem[Ch87]{Cheng}
J-H.~Cheng.
{\sl Graded Lie algebras of the Second kind.}
 Trans. Amer. Math. Soc. \textbf{302} (1987),  no.\,2, 467--488.

\bibitem[CoNSt]{fisicos}
L.~Corwin, Y.~Ne'eman  and S.~Sternberg.
{\sl   Graded Lie algebras in mathematics and physics (Bose-Fermi symmetry)}.
Rev. Mod. Phys. \textbf{47}  (1975), no.\,3, 573--603.

\bibitem[Cv]{libroinv}
 P.~Cvitanovi\'c. 
 Group theory. Birdtracks, Lie's, and exceptional groups. 
 Princeton University Press, Princeton, NJ, 2008. xiv+273 pp. 

\bibitem[Do09]{Dobrev}
V.K.~Dobrev.
\emph{Invariant Differential Operators for Non-Compact Lie Groups:  the  $E6(-14)$ case.}
 Proceedings, Eds. B. Dragovich, Z. Rakic, (Institute of Physics, Belgrade, SFIN Ser.
A: Conferences; A1 (2009)) pp. 95--124,
arXiv:0812.2655 [math-ph]

\bibitem[Do14]{Dobrev2}
V.K.~Dobrev.
\emph{Classification of Invariant Differential Operators for
Non-Compact Lie Algebras via Parabolic Relations.}
 J. Phys.: Conf. Ser. 512 012020

 \bibitem[DrE17]{nuestroe8}    
 C.~Draper and A.~Elduque.
 {\sl Maximal finite abelian subgroups of E8.}
  Proc. Roy. Soc. Edinburgh Sect. A \textbf{147} (2017), no.\,5, 993--1008. 

\bibitem[DrEl14]{proceedingsLecce}
C.~Draper and A.~Elduque.
 {\sl Fine gradings on the simple Lie algebras of type $E$}.
Note Mat. \textbf{34} (2014), no.\,1, 53--86.

 \bibitem[DrG16a]{Satakes}
C.~Draper and V.~Guido.
 \emph{On the real forms of the exceptional Lie algebra $\mathfrak e_6$ and their Satake diagrams}.
Non-associative and non-commutative algebra and operator theory, 211--226, 
Springer Proc. Math. Stat., \textbf{160}, Springer, Cham, 2016.

 \bibitem[DrG16b]{sig26}
 C.~Draper and V.~Guido.
 \emph{Gradings on the real form $\mathfrak{e}_{6,-26}$}. 
 J.~Math. Phys.   \textbf{57}(10) (2016), 18 pp. 

\bibitem[DrV16]{e6}
C.~Draper and A.~Viruel. 
\emph{Fine gradings on $\mathfrak{e}_6$}, 
Publ. Mat. \textbf{60}(1) (2016),   113--170.

\bibitem[DuF07]{fis7}
 M.J.~Duff and  S.~Ferrara. 
 {\sl $E_6$ and the bipartite entanglement of three qutrits.}
  Phys. Rev. D 76 (2007), no. 12, 124023, 7 pp.

\bibitem[EK]{libro}
A.~Elduque and M.~Kotchetov.
{\sl Gradings  on simple Lie algebras.}
Mathematical Surveys and Monographs 189, American
Mathematical Society, Providence, RI; Atlantic Association for Research
in the Mathematical Sciences (AARMS), Halifax, NS, 2013.

\bibitem[EK18]{reales1}
A.~Elduque and M.~Kotchetov.
{\sl  Gradings on the simple real Lie algebras of types G2 and D4}
Preprint arXiv:1803.10949.

\bibitem[GV05]{fis4}
T.~Gannon and M.~Vasudevan.
 JHEP 0507:035 (2005), hep-th/0504006.

 \bibitem[GPR18]{fibradosHiggs}
 O.~Garc\'\i a-Prada, A.~Pe\'on-Nieto and S.~Ramanan.
 {\sl Higgs bundles for real groups and the Hitchin-Kostant-Rallis section.}
  Trans. Amer. Math. Soc. \textbf{370} (2018), no.~4, 2907--2953.

\bibitem[Ge]{libroGeorgi}
H.~Georgi.
 {\sl  Lie algebras in particle physics. From isospin to unified theories.}
 Frontiers in Physics, 54.
 Benjamin/Cummings Publishing Co., Inc., Advanced Book Program, Reading, Mass., 1982. xxii+255 pp.

\bibitem[GLM07]{fis6}
S. Gurrieri, A. Lukas and A. Micu.
 JHEP 0712:081 (2007), arXiv:0709.1932.

 \bibitem[HPS08]{motivacion14}
M.~Henneaux, D.~Persson and P.~Spindel.
\emph{ Spacelike Singularities and Hidden Symmetries of Gravity},
Living Rev. Relativity \textbf{11} (2008), 1--232. 
http://www.livingreviews.org/lrr-2008-1

 \bibitem[HPPe00]{LGIII}  
 M.~Havl\'{\i}\v{c}ek, J.~Patera and E.~Pelantov\'{a}.
{\sl On Lie gradings III. Gradings of the real forms of classical
Lie algebras (dedicated to the memory of H. Zassenhaus)}. 
Linear Algebra   Appl. {\bf 314} (2000), 1--47.

\bibitem[H]{Draper:Humphreysalg}
J.E.~Humphreys. 
Introduction to Lie algebras and representation theory,
Graduate Texts in Mathematics \textbf{9},
Springer-Verlag, New-York, 1978.

\bibitem[HM06]{fis5}
P.Q.~Hung and P.~Mosconi.
 hep-ph/0611001.

 \bibitem[I97]{Iachello} 
 F.~Iachello.
{\sl Algebraic methods in physics.}
Symmetries in science, X (Bregenz, 1997), 135--143, Plenum, New York, 1998. 

 \bibitem[Ja]{Jacobsondeexcepcionales}
N.~Jacobson.
{\sl Exceptional Lie algebras.}
Lecture notes in pure and applied mathematics.
Marcel Dekker, Inc., New York, 1971.

\bibitem[KL05]{fis3}
J.~Kang and P.~Langacker.
 Phys. Rev. D71 (2005) 035014, hep-ph/0412190.

\bibitem[KGK10]{realquadrangle}
 T.~Kurth, R.~Gramlich and  L.~Kramer.
 {\sl The real quadrangle of type E6. }
 Adv. Geom. \textbf{10} (2010), no. 3, 505--526. 

\bibitem[MT16]{Marraniexcepcional}
A.~Marrani and P.~Truini. 
 {\sl Exceptional Lie algebras at the very foundations of space and time.}
  p-Adic Numbers Ultrametric Anal. Appl. \textbf{8} (2016), no. 1, 68--86.  

\bibitem[MPW91]{fis8}
 I.~Morrison, P.W.~Pieruschka and B.G.~Wybourne.
 {\sl  The interacting boson model with the exceptional groups G2 and E6.}
  J. Math. Phys. \textbf{32} (1991), no.\,2, 356--372.

\bibitem[O]{libroreales}
A.L.~Onishchnik.
{\sl Lectures on Real Semisimple Lie Algebras and Their Representations.}
European Mathematical Society, 2004.
Series ESI Lectures in Mathematics and Physics.

\bibitem[OV]{enci}
A.L.~Onishchnik, \`{E}.B.~Vinberg (Editors).
Lie Groups and  Lie Algebras III,  Encyclopaedia of Mathematical Sciences, Vol. 41.
Springer-Verlag, Berlin, 1991.

\bibitem[PZ89]{LGI} 
J.~Patera and H.~Zassenhaus.
 {\sl On Lie gradings. I}.
Linear Algebra  Appl.  \textbf{112} (1989), 87--159.

\bibitem[Sv08]{Svobodova}    
M.~Svobodov\'a.
{\sl Fine Gradings of Low-Rank Complex Lie Algebras and of Their Real Forms}.
SIGMA 4 (2008), 039, 13 pages.

\bibitem[T66]{Tits}
J.~Tits.
 {\sl Alg\`{e}bres alternatives, alg\`{e}bres de Jordan et alg\`{e}bres de Lie exceptionelles. I. Construction}. 
Nederl. Akab. Wetensch. 
Proc. Ser. A 69 = Indag. Math. \textbf{28} (1966), 223--237.

\bibitem[W95]{Wybourne}
B.G.~Wybourne. 
{\sl Exceptional Lie groups in physics.}
 Lithuanian Journal of Physics. Volume 35, 1995, no.\,2, 123--132.

\bibitem[Y16]{Yu}
J.~Yu. 
{\sl Maximal abelian subgroups of compact simple Lie groups of type E.}
 Geom. Dedicata  \textbf{185} (2016), 205--269. 


\end{thebibliography}
\end{document}